\documentclass[a4paper,11pt, twoside]{amsart}
\usepackage[utf8]{inputenc}
\usepackage{amssymb}
\usepackage{amsmath,amsthm}
\usepackage{etoolbox}
\usepackage{mathrsfs}
\usepackage[cal=boondox]{mathalfa}
\usepackage{amscd}
\usepackage{amsfonts}
\usepackage{tikz}
\usepackage{tikz-cd}
\usepackage[shortlabels]{enumitem}
\usepackage{rotating}
\usepackage{MnSymbol}
\usepackage{relsize}
\usepackage{bbm}
\usepackage{tcolorbox}
\usepackage{xcolor}
\usepackage{mathtools}
\usepackage[
backend=biber,
style=alphabetic,
sorting=anyt,
backref=true,
backrefstyle=none
]{biblatex}
\RequirePackage{doi}
\setlist{itemsep=1.0 pt}

\usepackage[a4paper,top=3cm,bottom=3cm,left=3cm,right=3cm]{geometry}
\usepackage{textgreek}

\theoremstyle{plain}
\newtheorem{theorem}{Theorem}[section]
\newtheorem{proposition}[theorem]{Proposition}
\newtheorem{lemma}[theorem]{Lemma}
\newtheorem{corollary}[theorem]{Corollary}

\theoremstyle{definition}

\newtheorem{defi}[theorem]{Definition}
\newtheorem{nota}[theorem]{Notation}
\newtheorem{ass}[theorem]{Assumption}

\newtheorem{ex/}[theorem]{Example}

\newtheorem{rmk/}[theorem]{Remark}
\newenvironment{rmk}
  {%
   \pushQED{\qed}\begin{rmk/}}
  {\popQED\end{rmk/}}

\newcommand{\bb}{\mathbb}
\newcommand{\frk}{\mathfrak}
\newcommand{\cl}{\mathcal}

\newcommand{\Hf}{{\boldsymbol{f}}}
\newcommand{\Hg}{{\boldsymbol{g}}}
\newcommand{\Hh}{{\boldsymbol{h}}}

\DeclareMathOperator{\et}{\mathrm{\acute{e}t}}

\DeclareMathOperator{\etcc}{\mathrm{\acute{e}t,c}}

\DeclareMathOperator{\dR}{\mathrm{dR}}

\DeclareMathOperator{\cris}{\mathrm{cris}}
\DeclareMathOperator{\Fil}{\mathrm{Fil}}
\DeclareMathOperator{\BK}{\mathrm{BK}}
\DeclareMathOperator{\Sel}{\mathrm{Sel}}

\DeclareMathOperator{\Ind}{Ind}

\DeclareMathOperator{\rel}{\mathrm{rel}}

\DeclareMathOperator{\cor}{cor}

\DeclareMathOperator{\Tsym}{Tsym}
\DeclareMathOperator{\Symm}{Symm}

\DeclareMathOperator{\pr}{\mathtt{pr}}
\DeclareMathOperator{\DET}{\mathtt{Det}}

\newcommand{\cW}{\mathcal W}

\newcommand{\Q}{\mathbb{Q}}
\newcommand{\Z}{\mathbb{Z}}

\newcommand{\C}{\mathbb{C}}

\newcommand{\Gal}{\mathrm{Gal\,}}
\newcommand{\GL}{\mathrm{GL}}

\newcommand{\Tr}{\mathrm{Tr}}

\newcommand{\cond}{\mathrm{cond}}
\newcommand{\Frob}{\mathrm{Frob}}
\newcommand{\End}{\mathrm{End}}

\newcommand{\Aut}{\mathrm{Aut}}
\newcommand{\Fr}{\mathrm{Fr}}

\newcommand{\ord}{{\mathrm{ord}}}
\newfont{\gotip}{eufb10 at 12pt}

\newcommand{\cO}{{\mathcal O}}

\newcommand{\SL}{{\mathrm {SL}}}

\newcommand{\hf}{{\mathbf{f}}}

\newcommand{\Ss}{\mathscr{S}}
\newcommand{\Ls}{\mathscr{L}}
\newcommand{\Fs}{\mathscr{F}}

\DeclareMathOperator{\Hom}{Hom}

\numberwithin{equation}{section}

\begin{document}

\title{Diagonal cycles and anticyclotomic twists of modular forms at inert primes}

\author{Luca Marannino}

\begin{abstract}
We revisit the construction of \cite{CD2023} of an anticyclotomic Euler system for the $p$-adic Galois representation of a modular form, using diagonal classes. Combining this construction and some of the results of \cite{Mar2024a}, we obtain new results towards the Bloch--Kato conjecture in analytic rank one, assuming that the fixed prime $p$ is inert in the relevant imaginary quadratic field.
\end{abstract}


\date{\today}

\address{Sorbonne Universit\'{e} and Universit\'{e} Paris Cit\'{e}, CNRS, IMJ-PRG, F-75005 Paris, France}
\email{marannino@imj-prg.fr}


\maketitle

\setcounter{tocdepth}{1}
\tableofcontents

\section{Introduction}
Let $f\in S_k(\Gamma_0(N_f))$ be a newform of even weight $k\geq 2$ and let $K$ be a quadratic imaginary field of discriminant $-D_K$ coprime to $N_f$. Assume that we can write $N_f=N_f^+N_f^-$ where $N_f^+$ (resp. $N_f^-$) is divisible only by primes inert (resp. split) in $K$ and with $N_f^-$ squarefree divisible by an \emph{odd} number of primes (this is known as the \emph{definite} setting). 

Let $\chi$ denote an anticyclotomic Hecke character of $K$ of infinity type $(-j,j)$ with $|j|\geq k/2$ and with conductor coprime to $D_K N_f$. In this setting, the sign of the functional equation for the complex $L$-function $L(f/K,\chi,s)$ is known to be $-1$. 

\medskip
Fix now a prime $p\nmid 6N_fD_K$, $p$ coprime to the conductor of $\chi$. The Bloch-Kato conjecture predicts that the Bloch-Kato Selmer group for the $p$-adic $G_K:=\Gal(\bar{K}/K)$-representation  
\[
V_{f,\chi}:=V_f(1-k/2)\otimes\hat{\chi}^{-1}
\]
should have generically dimension $1$ (over a large enough finite extension $E$ of $\Q_p$ that we fix as our field of coefficients). Here $V_f$ denotes the (dual of the) $p$-adic Galois representation attached to $f$ by Deligne and $\hat{\chi}$ denotes the $G_K$-character attached via (geometrically normalized) class field theory to the so-called $p$-adic avatar of $\chi$.

\subsection{Main results}
Under suitable technical assumptions, the main result of this note gives a $p$-adic criterion to ensure that the Bloch-Kato Selmer group $\Sel_{\BK}(K,V_{f,\chi})$ is indeed one dimensional, under the hypothesis that the prime $p$ is inert in $K$. This criterion is expressed in terms of the $p$-adic $L$-function $\Theta_\infty^{\mathrm{Heeg}}(f,\chi_t)$ constructed by Chida-Hsieh in \cite{CH2018} (generalizing a construction of \cite{BD1996} for $f$ of weight $2$). Here $\chi_t$ denotes the \emph{tame part} of $\hat{\chi}$, which is a ring class character (see Section \ref{factorizationsectionintro} for more details). The square of $\Theta_\infty^{\mathrm{Heeg}}(f,\chi_t)$ interpolates the central values $L(f/K,\chi_t\nu,k/2)$ for $\nu$ an anticyclotomic character of $K$ of $p$-power conductor and $\infty$-type $(-t,t)$ with $|t|<k/2$. We can view $\Theta_\infty^{\mathrm{Heeg}}(f,\chi_t)$ as an element of the Iwasawa algebra $\cl{O}_E[\![\Gamma^-]\!]$, where $\Gamma^-\cong\Z_p$ is the Galois group of the anticyclotomic $\Z_p$-extension of $K$ and we can evaluate it at continuous characters $\Gamma^-\to\C_p^\times$ ($\C_p$ being the completion of a fixed algebraic closure of $\Q_p$). In particular, one can evaluate it at $\chi^-$, the \emph{wild part} of $\hat{\chi}$ (cf. again Section \ref{factorizationsectionintro} for the notation).

\begin{ass}
\label{introass}
Here we group the extra technical hypothesis needed for the proof of the main result of this paper:
\begin{enumerate}[(i)]
    \item $f$ is not a CM form and has big image at $p$ (cf. Definition \ref{def: big image});
    \item The mod $p$ Galois representation attached to $f$ is absolutely irreducible and $p$-distinguished.
    \item $p\nmid h_K$, where $h_K$ is the class number of $K$;
    \item $p>\max\{k-2,j+1\}$;
    \item $N_f$ is squarefree (i.e., $N_f^+$ is squarefree too);
    \item $\chi$ has conductor $c\cdot\cl{O}_K$ where $c\in\Z_{\geq 1}$ is coprime to $pD_K N_f$ and divisible only by primes split in $K$ and satisfies Assumption \ref{asschit} below.
\end{enumerate}
\end{ass}

As pointed out in Remark \ref{rmk: notseriousass} below, assumptions (v) and (vi) can certainly be relaxed with some more work.

\begin{theorem}
\label{mainThmIntro}
In the setting described above (including the hypothesis in Assumption \ref{introass}), we have that
\[
\chi^-(\Theta_\infty^{\mathrm{Heeg}}(f,\chi_t))\neq 0 \quad\Rightarrow \quad \dim_E\Sel_{\BK}(K,V_{f,\chi})=1
\]
\end{theorem}

\medskip
The proof of Theorem \ref{mainThmIntro} relies on the anticyclotomic Euler system constructed in \cite{CD2023} starting from diagonal classes. This Euler system is of the kind introduced by Jetchev-Nekov\'{a}\v{r}-Skinner in forthcoming work (we refer to \cite[Section 8.1]{ACR2023} for some of the features of this theory). In Section \ref{sec: JNSES} we review the construction of the Euler system (in the case of general balanced triples of weights). Since we assume that $p$ is inert, we are not able to construct a full Euler system (we cannot obtain classes in the vertical direction). Still, under suitable \emph{big image} assumptions, the non-vanishing of the bottom class $\kappa_{f,\chi}$ implies the one-dimensionality of the Bloch-Kato Selmer group.

\medskip
In Sections \ref{sec: padicL} and \ref{section: ERL} we introduce the ingredients that allow to prove that the non-vanishing of $\kappa_{f,\chi}$ is implied by the non-vanishing of the $p$-adic $L$-values $\chi^-(\Theta_\infty^{\mathrm{Heeg}}(f,\chi_t))$. Section \ref{sec: padicL} recollects some results obtained in \cite{Mar2024a} on generalized triple product $p$-adic $L$-functions. Section \ref{section: ERL} elaborates on the explicit reciprocity laws for diagonal classes proven in \cite{BSV2020a}. Section \ref{sec: final proof} is finally devoted to conclude the proof of Theorem \ref{mainThmIntro}.

\subsection{Relation with previous works and future perspectives}
As already indicated, this work deeply relies on \cite{CD2023}, in which the authors construct the relevant Euler system and study the applications to the case in which $p$ is split in $K$, obtaining stronger results both in the definite and indefinite setting. 

\medskip
We refer to the Introduction of \cite{CD2023} for a brief account on the literature on the subject and for a justification of the fact that one cannot rely on (generalized) Heegner points and/or Beilinson-Flach classes to deal with the setting studied in our work.

\medskip
For a more detailed survey on previous results on the Bloch-Kato conjecture and Iwasawa main conjecture for anticyclotomic twists of modular forms (with a specific emphasis on the $p$-inert case) we refer to the first chapter of \cite{BL2023}. We actually hope that some of the results of \cite{BL2023} towards Iwasawa main conjectures in the $p$-inert case could be reproved or improved relying on the anticyclotomic Euler system studied in this work.

\subsection*{Acknowledgements}
The author is currently funded by the Simons Collaboration on Perfection in Algebra, Geometry and Topology as postdoctoral researcher at CNRS. The author wishes to thank Francesc Castella and Kim Tuan Do for sharing an updated version of \cite{CD2023} and for spotting some mistakes in a first version of this work and Massimo Bertolini for the encouragement.

\section{Modular curves and their cohomology}\label{subsec:prelim}

In this section we fix our conventions for modular curves and Hecke operators. We refer to \cite[section 2]{ACR2023} and \cite[section 2]{BSVast1} (and the references therein) for more details.
\subsection{Modular curves}

Given integers  $M\geq 1$, $N\geq 1$, $m\geq 1$ and $n\geq 1$ with
$M+N\geq 5$, we denote by $Y(M(m),N(n))$ the affine modular curve over $\bb{Z}[1/MNmn]$ representing the functor taking a $\bb{Z}[1/MNmn]$-scheme $S$ to the set of isomorphism classes of 5-tuples $(E,P,Q,C,D)$, where:
\begin{enumerate}[(i)]
\item $E$ is an elliptic curve over $S$,
\item $P$ is an $S$-point of $E$ of order $M$,
\item $Q$ is an $S$-point of $E$ of order $N$,
\item $C$ is a cyclic order-$Mm$ subgroup of $E$ defined over $S$ and containing $P$,
\item $D$ is a cyclic order-$Nn$ subgroup of $E$ defined over $S$ and containing $Q$,
\end{enumerate}
and such that $C$ and $D$ have trivial intersection. If $m=1$ or $n=1$, we omit it from the notation. As usual, we will often write $Y_1(N)$ for $Y(1,N)$.

We denote by
\[
v=v_{M(m),N(n)}:E(M(m),N(n))\rightarrow Y(M(m),N(n))
\]
the universal elliptic curve over $Y(M(m),N(n))$.

\medskip
Define the modular group
\begin{align*}
\Gamma & (M(m),N(n)):=\\
&\Big\{\begin{psmallmatrix} a & b \\ c & d\end{psmallmatrix}\in \text{SL}_2(\bb{Z}): a\equiv 1\!\!\mod M, b\equiv 0\!\!\mod Mm, c\equiv 0\!\!\mod Nn,d\equiv 1\!\!\mod N\Big\}.  
\end{align*}
Then, denoting by $\mathcal{H}$ the upper half-plane, we have the complex uniformization
\begin{equation}\label{eq:unif}
Y(M(m),N(n))(\bb{C})\cong (\bb{Z}/M\bb{Z})^\times\times\Gamma(M(m),N(n))\backslash\mathcal{H},
\end{equation}
with a pair $(a,\tau)$ on the right-hand side corresponding
to the isomorphism class of the 5-tuple $(\bb{C}/\bb{Z}+\bb{Z}\tau, a\tau/M,1/N,\langle \tau/Mm\rangle, \langle 1/Nn\rangle)$.

If $r\geq 1$ is an integer, there is an isomorphism of $\bb{Z}[1/MNmnr]$-schemes
\[
\varphi_r: Y(M(m),N(nr))\xrightarrow{\simeq} Y(M(mr),N(n))
\]
defined by $(E,P,Q,C,D)\mapsto (E', P', Q', C',D')$, where $E'=E/NnD$, $P'$ is the image of $P$ in $E'$, $Q'$ is the image of $r^{-1}(Q)\cap D$ in $E'$, $C'$ is the image of $r^{-1}(C)$ in $E'$, and $D'$ is the image of $D$ in $E'$. Under the complex uniformizations $(\ref{eq:unif})$, the isomorphism $\varphi_r$ sends $(a,\tau)\mapsto (a,r\cdot\tau)$. If
\[
\varphi_r^\ast(E(M(mr),N(n)))\rightarrow Y(M(m),N(nr))
\]
denotes the base change of $E(M(mr),N(n))\rightarrow Y(M(mr),N(n))$ under $\varphi_r$, there is a natural degree-$r$ isogeny
\[
\lambda_r:E(M(m),N(nr))\rightarrow \varphi_r^\ast(E(M(mr),N(n))).
\]

\subsection{Degeneracy maps}

With the same notations as above, we have natural degeneracy maps
\begin{align*}
& Y(M(m),Nr(n))\xrightarrow{\mu_r} Y(M(m),N(nr))\xrightarrow{\nu_r} Y(M(m),N(n)), \\
& Y(Mr(m),N(n))\xrightarrow{\check{\mu}_r} Y(M(mr),N(n))\xrightarrow{\check{\nu}_r} Y(M(m),N(n)),
\end{align*}
forgetting the extra level structure, i.e.,
\begin{align*}
& \mu_r(E,P,Q,C,D)=(E,P,r\cdot Q, C,D), \\
& \nu_r(E,P,Q,C,D)=(E,P,Q,C,rD).
\end{align*}
and $\check{\mu}_r$ and $\check{\nu}_r$ defined analogously.

We also define degeneracy maps
\begin{equation}\label{eq:pi12}
\begin{aligned}
& \pi_1: Y(M(m),Nrs(nt))\rightarrow Y(M(m),N(ns)), \\
& \pi_2: Y(M(m),Nrs(nt))\rightarrow Y(M(m),N(ns)),
\end{aligned}
\end{equation}
determined by
\begin{align*}
& \pi_1(E,P,Q,C,D)=(E,P,rs\cdot Q,C, rtD), \\
& \pi_2(E,P,Q,C,D)=(E',P',Q',C',D'),
\end{align*}
where $E'=E/NnsD$, $P'$ is the image of $P$ in $E'$, $Q'$ is the image of $t^{-1}(s\cdot Q)\cap D$ in $E'$, $C'$ is the image of $C$ in $E'$ and $D'$ is the image of $D$ in $E'$. Under the complex uniformizations (\ref{eq:unif}), the map $\pi_1$ (resp. $\pi_2$) corresponds to the identity (resp. to multiplication by $rt$) on $\mathcal{H}$. It is straightforward to check that the maps $\pi_1$ and $\pi_2$ are respectively given by the compositions
\begin{align*}
& Y(M(m),Nrs(nt))\xrightarrow{\mu_{rs}} Y(M(m),N(nrst))\xrightarrow{\nu_{rt}} Y(M(m),N(ns)), \\
& Y(M(m),Nrs(nt))\xrightarrow{\mu_{rs}} Y(M(m),N(nrst))\xrightarrow{\varphi_{rt}} Y(M(mrt),N(ns))\xrightarrow{\check{\nu}_{rt}} Y(M(m),N(ns)).
\end{align*}

Finally, we will often need to consider the following specific degeneracy maps in what follows, which justifies introducing \emph{ad hoc} notation. Fix $N\in\Z_{\geq 5}$ and let $q\geq 1$ be a rational prime. We will typically rename:
\begin{enumerate}[(i)]
    \item $\pr_1:Y_1(Nq)\to Y_1(N)$ the specific instance of the map $\pi_1$ in \eqref{eq:pi12};
    \item $\pr_q:Y_1(Nq)\to Y_1(N)$ the specific instance of the map $\pi_2$ in \eqref{eq:pi12};
    \item $\pr_{11}:Y_1(Nq^2)\to Y_1(N)$ the specific instance of the map $\pi_1$ in \eqref{eq:pi12};
    \item $\pr_{qq}:Y_1(Nq^2)\to Y_1(N)$ the specific instance of the map $\pi_2$ in \eqref{eq:pi12}.
\end{enumerate}

\subsection{Relative Tate modules}\label{subsubsec:relT}

Fix a prime $p$. Given a $\bb{Z}[1/MNmnp]$-algebra $R$ and a $\bb{Z}[1/MNmnp]$-scheme $X$, we write $X_R$ for the base change of $X$ to $R$. For any $\bb{Z}[1/MNmnp]$-scheme $X$, denote by $A=A_X$ either the locally constant constructible sheaf $\mathbb Z/p^t(j)$ or the locally constant $p$-adic sheaf $\mathbb Z_p(j)$ on $X_{\et}$, for fixed $t\geq 1$ and $j\in\Z$. Write 
$$
v_R=v_{M(m),N(n),R}:E(M(m),N(n))_R\rightarrow Y(M(m),N(n))_R
$$
for the base change to a $\bb{Z}[1/MNmnp]$-algebra $R$ of the universal elliptic curve and set
$$
\mathscr{T}^*_{M(m),N(n)}(A):=R^1v_{R,\ast} \bb{Z}_p\otimes_{\bb{Z}_p} A\quad\text{and}\quad \mathscr{T}_{M(m),N(n)}(A):=\Hom(\mathscr{T}^*_{M(m),N(n)}(A),A).
$$
In particular, when $A=\bb{Z}_p$, this gives the dual of the relative Tate module and (respectively) the relative Tate module of the universal elliptic curve; in this case, we will often drop $A$ from the notation.

\medskip
The proper base change theorem for \'{e}tale cohomology ensures that $\mathscr T_{M(m),N(n)}(A)$ and $\mathscr T_{M(m),N(n)}^*(A)$ are locally constant $p$-adic sheaves on $Y(M(m),N(n))_S$ whose formation is compatible with base change along morphisms of $\bb{Z}[1/MNmnp]$-algebras $R\rightarrow R'$.

\medskip
For every integer $r\geq 0$, define
\[
\Ls_{M(m),N(n),r}(A) = \Tsym_A^r \mathscr T_{M(m),N(n)}(A), \quad \Ss_{M(m),N(n),r}(A) = \Symm_A^r \mathscr T_{M(m),N(n)}^*(A),
\]
where, for any finite free module $M$ over a profinite $\mathbb Z_p$-algebra $R$, one denotes by $\Tsym_R^r M$ the $R$-submodule of symmetric tensors in $M^{\otimes r}$ and by $\Symm_R^r M$ the maximal symmetric quotient of $M^{\otimes r}$. If the the level of the modular curve is clear, we will use the simplified notations \[ \Ls_r(A) = \Ls_{M(m),N(n),r}(A), \quad \Ls_r = \Ls_r(\mathbb Z_p), \quad \Ss_r(A) = \Ss_{M(m),N(n),r}(A), \quad \Ss_r = \Ss_r(\mathbb Z_p). \]

\subsection{Hecke and diamond operators}\label{subsec:hecke}

Let $\Fs_{M(m),N(n)}^r$ denote either $\Ls_{M(m),N(n),r}(A)$ or $\Ss_{M(m),N(n),r}(A)$ and let $q$ be a rational prime. Then there are natural isomorphisms of sheaves
\begin{equation}\label{eq:iso-q}
\nu_q^\ast(\Fs_{M(m),N(n)}^r)\cong \Fs_{M(m),N(nq)}^r\quad\text{and}\quad \check{\nu}_q^\ast(\Fs_{M(m),N(n)}^r)\cong \Fs_{M(mq),N(n)}^r,
\end{equation}
and therefore pullback morphisms
\begin{align*}
& H^i_{\et}(Y(M(m),N(n))_R,\Fs_{M(m),N(n)}^r)\xrightarrow{\nu_q^\ast} H^i_{\et}(Y(M(m),N(nq))_R,\Fs_{M(m),N(nq)}^r), \\
& H^i_{\et}(Y(M(m),N(n))_R,\Fs_{M(m),N(n)}^r)\xrightarrow{\check{\nu}_q^\ast} H^i_{\et}(Y(M(mq),N(n))_R,\Fs_{M(mq),N(n)}^r),
\end{align*}
and traces
\begin{equation}\label{eq:trace}
\begin{aligned}
& H^i_{\et}(Y(M(m),N(nq))_R,\Fs_{M(m),N(nq)}^r)\xrightarrow{\nu_{q\ast}} H^i_{\et}(Y(M(m),N(n))_R,\Fs_{M(m),N(n)}^r), \\
& H^i_{\et}(Y(M(mq),N(n))_R,\Fs_{M(mq),N(n)}^r)\xrightarrow{\check{\nu}_{q\ast}} H^i_{\et}(Y(M(m),N(n))_R,\Fs_{M(m),N(n)}^r).
\end{aligned}
\end{equation}

Also, the isogeny $\lambda_q$ induces morphisms of sheaves
$$
\lambda_{q\ast} : \Fs_{M(m),N(nq)}^r\rightarrow \varphi_q^\ast(\Fs_{M(mq),N(n)}^r) \quad\text{and}\quad
\lambda_{q}^\ast : \varphi_q^\ast(\Fs_{M(mq),N(n)}^r)\rightarrow \Fs_{M(m),N(nq)}^r.
$$
These morphisms allow us to define
\begin{align*}
& \Phi_{q\ast}:H^i_{\et}(Y(M(m),N(nq))_R,\Fs_{M(m),N(nq)}^r)\rightarrow H^i_{\et}(Y(M(mq),N(n))_R,\Fs_{M(mq),N(n)}^r), \\
& \Phi_{q}^\ast:H^i_{\et}(Y(M(mq),N(n))_R,\Fs_{M(mq),N(n)}^r)\rightarrow H^i_{\et}(Y(M(m),N(nq))_R,\Fs_{M(m),N(nq)}^r),
\end{align*}
as the compositions
\[
\Phi_{q\ast}=\varphi_{q\ast}\circ\lambda_{q\ast}\quad\text{and}\quad \Phi_q^\ast=\lambda_q^\ast\circ\varphi_q^\ast.
\]
We define the Hecke operators $T_q$ and the adjoint Hecke operators $T_q'$ acting on the \'{e}tale cohomology groups
\[
H^i_{\et}(Y(M(m),N(nq))_R,\Fs_{M(m),N(n)}^r)
\]
as the compositions $
T_q=\check{\nu}_{q\ast}\circ\Phi_{q\ast}\circ\nu_q^\ast$ and $T_q'=\nu_{q\ast}\circ\Phi_q^\ast\circ\check{\nu}_q^\ast$.

Now we introduce diamond operators. For $d\in (\bb{Z}/MN\bb{Z})^\times$, these are defined on the curves
$Y(M(m),N(n))$ as the automorphisms $\langle d\rangle$ acting on the moduli space by
$$
(E,P,Q,C,D)\mapsto (E,d^{-1}\cdot P,d\cdot Q, C, D).
$$
We can also define the diamond operator $\langle d \rangle$ on the corresponding universal elliptic curve as the unique automorphism making the diagram
\begin{center}
\begin{tikzcd}
&E(M(m),N(n))_R\arrow[r, "\langle d\rangle"]\arrow[d, "v"] & E(M(m),N(n))_R\arrow[d, "v"] \\ 
&Y(M(m),N(n))_R\arrow[r, "\langle d\rangle"] & Y(M(m),N(n))_R
\end{tikzcd}
\end{center}
cartesian. This in turn induces automorphisms $\langle d\rangle = \langle d\rangle^\ast$ and $\langle d\rangle'=\langle d\rangle_\ast$ on the group $H^i_{\et}(Y(M(m),N(n))_R,\Fs_{M(m),N(n)}^r)$ which are inverses of each other.

We will focus mostly on modular curves of the form $Y(1(m),N(n))$. In this case, the natural pairing $\Ls_r\otimes_{\bb{Z}_p}\Ss_r\rightarrow \bb{Z}_p$ together with cup-product yields a pairing
\begin{equation}
\label{Pkduality}
H^1_{\et}(Y(1(m),N(n))_R,\Ls_r(1))\otimes_{\bb{Z}_p} H^1_{\et,c}(Y(1(m),N(n))_R,\Ss_r)\rightarrow \bb{Z}_p    
\end{equation}
which becomes perfect after inverting $p$. The operators $T_q$, $T_q'$, $\langle d\rangle$, $\langle d\rangle'$ induce endomorphisms on compactly supported cohomology and
$$
(T_q,T_q'),\quad (T_q',T_q),\quad (\langle d\rangle,\langle d\rangle') \quad\text{and}\quad (\langle d\rangle',\langle d\rangle)
$$
are adjoint pairs under this pairing.

\medskip
As explained in \cite[Section 3]{BSVast1}, the determinant induces an isomorphism of sheaves
\begin{equation}
\label{isosr}
\mathbf{s}_r:\Ss_r(\Q_p)\cong \Ls_r(\Q_p)\otimes_{\Z_p}\Z_p(-r).
\end{equation}
As explained in \cite[Remark 3.3]{BSVast1}, the above isomorphism introduces denominators that are bounded by $r!$. 

In particular, if $p>r$, the (images of the) classes in $H^i(Y(1,N(n))_R,\Ss_r)$ are sent to (images of) classes in $H^i(Y(1,N(n))_R,\Ls_r(-r))$ under the induced isomorphism
\[
\mathbf{s}_r:H^i(Y(1,N(n))_R,\Ss_r(\Q_p))\xrightarrow{\cong}H^i(Y(1,N(n))_R,\Ls_r(\Q_p)(-r)).
\]
\subsection{Galois representations}
\label{subsec: Galoisreps}
From now on in this section $E$ will denote a finite (and large enough) extension of $\Q_p$ with ring of integers $\cO$.

\begin{nota}
\label{Heckeisotypicalcomp}
If $V$ is a $\Q_p$-vector space, we write $V_E:=V\otimes_{\Q_p}E$ and if $M$ is a $\Z_p$-module, we write $M_E:=M\otimes_{\Z_p}E$ and $M_\cO=M\otimes_{\Z_p}\cO$.

For an $E$-vector space $H$ endowed with an action of the good Hecke operators (i.e., $T_q$ for $q\nmid N$) and diamond operators of level $N$ and for $\xi\in S_\nu(N,\chi,E)$ an eigenform, we let $H[\xi]$ denote the $\xi$-Hecke isotypical component of $H$ (i.e., the maximal $E$-vector subspace where the Hecke and diamond operators act with the same eigenvalues as on $\xi$).  
\end{nota}

\begin{defi}
Given $\xi\in S_\nu(N,\chi,E)$ with $\nu\geq 2$ a normalized eigenform with Fourier coefficients lying in $E$, one can attach to it two Galois representations.
\begin{itemize}
    \item [(i)] We let $V_{N}(\xi)$ be the maximal $E$-quotient of $H^1_{\et}(Y_1(N)_{\bar{\Q}},\Ls_{\nu-2}(1))_E$ where the dual good Hecke operators (i.e., $T'_q$ for $q\nmid N$) and dual diamond operators act with the same eigenvalues as those of $\xi$. We also let $T_N(\xi)\subset V_N(\xi)$ denote the lattice obtained as the image of $H^1_{\et}(Y_1(N)_{\bar{\Q}},\Ls_{\nu-2}(1))_\cO$ inside $V_N(\xi)$ under the quotient map.
    \item [(ii)] We let $V^*_{N}(\xi):=H^1_{\etcc}(Y_1(N)_{\bar{\Q}},\Ss_{\nu-2})_E[\xi]$ (cf. Notation \ref{Heckeisotypicalcomp}) be the maximal $E$-subspace of $H^1_{\etcc}(Y_1(N)_{\bar{\Q}},\Ss_{\nu-2})_E$ where good Hecke and diamond operators act with the same eigenvalues as those of $\xi$.
    \end{itemize}
When the level $N$ is understood, we simply write $V(\xi)=V_N(\xi)$ and $V^*(\xi)=V^*_N(\xi)$.
\end{defi}

\begin{rmk}
If $\xi$ is new of level $N$, then $V^*_N(\xi)$ is identified with the $p$-adic Deligne representation associated to $\xi$ and $V_N(\xi)$ is identified with its dual. This means that $V_N(\xi)$ affords the unique (up to isomorphism) two-dimensional representation
\[
\rho_\xi\,:\,G_\bb{Q}\longrightarrow \GL_2(E)
\]
unramified outside $pN$ and characterized by the property that
\[
{\rm trace}\,\rho_\xi(\Fr_{q})=a_{q}(\xi)
\]
for all primes $q\nmid pN$, where $\Fr_{q}$ denotes an arithmetic Frobenius element at $q$.

In general $V_N(\xi)$ (resp. $V_N^*(\xi)$) is non-canonically isomorphic to finitely many copies of $V_{N_\xi}(\xi^{\circ})$ (resp. $V^*_{N_\xi}(\xi^{\circ})$), where $\xi^{\circ}$ is the newform of level dividing $N$ associated to $\xi$ and $N_\xi\mid N$ is the corresponding level.
\end{rmk}

\section{The construction of the Euler system}
This section recalls the construction of the anticyclotomic Euler systems appearing in \cite{CD2023}. We decided to write down some details in the construction because we need a direct construction for general balance triples of weights (not being able to rely on $p$-adic interpolation in the $p$-inert case) and because we think that it could still be interesting to compare the two slightly different constructions.

\label{sec: JNSES}
\subsection{Patching CM modules}\label{subsec:llz}

In this section we extend some results of \cite[Section 5.2]{LLZ2015} to the case of higher weight modular forms and we develop a bit of the machinery that is needed for the construction of the Euler systems.




Let $K$ be an imaginary quadratic field of discriminant $-D_K<0$, and let $\varepsilon_K$ be the corresponding quadratic character. Let $\psi$ be a Hecke character of $K$ of infinity type $(1-\nu,0)$ for some integer $\nu
\in\Z_{\geq 2}$ and conductor $\mathfrak f$, taking values in a finite extension $L/K$, and let $\chi$ be the unique Dirichlet character modulo $N_{K/\mathbb Q}(\mathfrak f)$ such that $\psi((n))=n^{\nu-1} \chi(n)$ for  integers $n$ coprime to $N_{K/\mathbb Q}(\mathfrak f)$. Put $N_\psi=N_{K/\mathbb Q}(\mathfrak f)D_K$ and $\chi_\psi=\chi\cdot\varepsilon_K$, and let $\theta_{\psi} \in S_\nu(N_{\psi}, \chi_\psi)$ be the newform attached to $\psi$, i.e.,
\[ \theta_{\psi} = \sum_{(\mathfrak a, \mathfrak f)=1} \psi(\mathfrak a) q^{ N_{K/\mathbb Q}(\mathfrak a)}.
\]

Fix a prime $p\geq 5$ unramified in $K$ and coprime to the class number of $K$, a prime $\frk{p}$ of $K$ above $p$ and a prime $\mathfrak{P}$ of $L$ above $\frk{p}$ (induced by the fixed embdedding $\iota_p$). Let $E=L_{\mathfrak P}$ be the completion of $L$ at $\mathfrak{P}$, with ring of integers $\cl{O}$. Let $\hat{\psi}$ be the $p$-adic avatar of $\psi$, i.e., the continuous $E$-valued character of $K^{\times} \backslash \mathbb A_K^{(\infty),\times}$ defined by 
\[ 
\hat{\psi}(x)=x_{\mathfrak p}^{1-\nu} \psi(x), 
\] 
where $x_{\mathfrak p}$ is the projection of the id\`ele $x$ to the component at $\mathfrak p$. We will also denote by $\hat{\psi}$ the corresponding character of $G_K$ obtained via class field theory (geometrically normalized). Then $V_{N_\psi}(\theta_\psi)\cong\Ind_K^\bb{Q} E(\hat{\psi}^{-1})$.

\begin{defi}
For an integral ideal $\frk{n}$ of $K$, we denote by $H(\frk{n})$ the maximal $p$-quotient of the corresponding ray class group, and by $K(\frk{n})$ the maximal $p$-extension in the corresponding ray class field. When $\frk{n}=n\cl{O}_K$ for $n\in\Z_{\geq 1}$, we simply write $H(n)$ and $K(n)$. We similarly define $H[n]$ and $K[n]$, for each integer $n>0$, as the maximal $p$-quotient in the corresponding ring class group and the maximal $p$-extension in the corresponding ring class field.
\end{defi}

Let $\mathfrak n$ be an integral ideal of $K$ and let $N=N_{K/\mathbb Q}(\mathfrak n)N_\psi$. Let $\bb{T}'_\nu(N)$ be the subalgebra of $\End_{\Z_p}\big(H^1_{\et}(Y_1(N)_{\bar{\Q}},\Ls_{\nu-2}(1))\big)$ generated by all the Hecke operators $T_q'$ (also for $q\mid N$) and diamond operators $\langle d\rangle'$ for $d\in(\Z/N\Z)^\times$.

\begin{proposition}
\label{prop:LLZ}
There is a morphism $\phi_{\frk{n}}:\bb{T}'_\nu(N) \rightarrow\cl{O}[H(\frk{n}\frk{f})]$ of $\Z_p$-algebras defined on generators by
\[
\phi_\frk{n}(T_{q}') = \sum_{\frk{q}}\psi(\frk{q})[\frk{q}]\qquad\text{and}\qquad \phi_\frk{n}(\langle d\rangle')=\chi_\psi(d)[(d)],
\]
where the (possibly empty) sum runs over ideals coprime to $\frk{n}\frk{f}$ of norm $q$.
\begin{proof}
If $\nu=2$, this is essentially \cite[Proposition 3.2.1]{LLZ2015}. The generalization to higher weights is obvious.
\end{proof}
\end{proposition}

Now let $\frk{n'}=\frk{nq}$ for some prime ideal $\frk{q}$ above a rational prime $q$ and let $N'=N_{K/\bb{Q}}(\frk{n'})N_\psi$. Following \cite[Section 3.3]{LLZ2015}, we define norm maps
\[
\cl{O}[H(\frk{n}')]\otimes_{(\bb{T}'_\nu(N'),\phi_\frk{n'})} H^1_{\et}(Y_1(N')_{\overline{\bb{Q}}},\Ls_{\nu-2}(1))\xrightarrow{\cl{N}_\frk{n}^\frk{n'}} \cl{O}[H(\frk{n})]\otimes_{(\bb{T}'_\nu(N),\phi_\frk{n})} H^1_{\et}(Y_1(N)_{\overline{\bb{Q}}},\Ls_{\nu-2}(1))
\]
by the formulae:
\begin{enumerate}[(i)]
\item if $\frk{q}\mid\frk{n}$,
$$
\cl{N}_\frk{n}^\frk{n'}:=1\otimes \pr_{1\ast};
$$
\item if $\frk{q}\nmid\frk{n}$ and $q\neq p$ is ramified or split,
$$
\cl{N}_\frk{n}^\frk{n'}:=1\otimes \pr_{1\ast}-\frac{\psi(\frk{q})[\frk{q}]}{q^{\nu-1}}\otimes \pr_{q\ast};
$$
\item if $\frk{q}\nmid\frk{n}$ and $q\neq p$ is inert,
$$
\cl{N}_\frk{n}^\frk{n'}:=1\otimes \pr_{11\ast}+\frac{\chi_{\psi}(q)[(q)]}{q^{\nu-1}}\otimes \pr_{qq\ast}.
$$
\end{enumerate}

\begin{rmk}
In general, one cannot define $\cl{N}_\frk{n}^\frk{np}$, unless $\frk{p}\mid\frk{n}$ (due to the presence of denominators). On the other hand, if one assumes that $p$ splits in $K$ as $(p)=\frk{p}\bar{\frk{p}}$ (soon this will not be the case for us) and $\bar{\frk{p}}\nmid\frk{n}$, one can still define $\cl{N}_\frk{n}^\frk{np}$ (cf. \cite[Proposition 4.3.6]{LLZ2015} and \cite[Proposition 5.2.5]{LLZ2015}).
\end{rmk}

More generally, for $\frk{n'}=\frk{n}\frk{r}$ with $\frk{r}$ a product of (not necessarily distinct) prime ideals, we define the map $\cl{N}_\frk{n}^\frk{n'}$ by composing in the natural way the previously defined norm maps (if such norm maps can be defined). 

\begin{theorem}
\label{thm:LLZ}
Assume that $p>\nu$ and that if $p$ splits (resp. is inert) in $K$, then $\bar{\frk{p}}\nmid\frk{f}$ (resp. $(p)\nmid\frk{f}$). Let $\cl{B}_{\frk{f}}$ be the set of integral ideals of $K$ coprime to $\overline{\frk{p}}\cdot (\nu-2)!$ (resp. $p\cdot (\nu-2)!$) if $p$ splits (resp. is inert) in $K$. Then there is a family of $G_\bb{Q}$-equivariant isomorphisms of $\cl{O}[H(\frk{n}\frk{f})]$-modules
\[
\nu_{\frk{n}}: H^1_{\et}(Y_1(N)_{\overline{\bb{Q}}},\Ls_{\nu-2}(1))\otimes_{(\bb{T}'_\nu(N),\phi_{\frk{n}})}\cl{O}[H(\frk{n}\frk{f})]\xlongrightarrow{\cong}\Ind_{K(\frk{n}\frk{f})}^\bb{Q}\cl{O}(\hat{\psi}^{-1})
\]
for all integral ideals $\frk{n}\in\cl{B}_{\frk{f}}$, such that for $\frk{n}\mid\frk{n'}$ the diagram
\begin{center}
\begin{tikzcd}
& H^1_{\et}(Y_1(N')_{\overline{\bb{Q}}},\Ls_{\nu-2}(1))\otimes_{(\bb{T}'_\nu(N'),\phi_\frk{n}')}\cl{O}[H(\frk{n}'\frk{f})]\arrow[r, "\nu_{\frk{n}'}", "\cong"']\arrow[d,"\cl{N}_\frk{n}^{\frk{n}'}"] &\Ind_{K(\frk{n'}\frk{f})}^\bb{Q}\cl{O}(\hat{\psi}^{-1})\arrow[d]\\
&H^1_{\et}(Y_1(N)_{\overline{\bb{Q}}},\Ls_{\nu-2}(1))\otimes_{(\bb{T}'_\nu(N),\phi_\frk{n})}\cl{O}[H(\frk{n}\frk{f})]\arrow[r, "\nu_{\frk{n}}", "\cong"']\arrow[r,"\nu_{\frk{n}}", "\cong"' ]&\Ind_{K(\frk{n}\frk{f})}^\bb{Q}\cl{O}(\hat{\psi}^{-1})
\end{tikzcd}
\end{center}
commutes, where the right vertical arrow is the natural norm map.
\begin{proof}
When $\nu=2$ this is \cite[Corollary 5.2.6]{LLZ2015}. The latter result relies on Ihara's lemma (cf. \cite[Lemma 4.2.1]{LLZ2015} for the statement) and on a crucial theorem due to Wiles (cf. \cite[Theorem 4.1.4]{LLZ2015}). Ihara's lemma has been generalized to higher weight forms in \cite[Lemma 3.2]{Dia1991} (which forces to exclude ideals dividing $(\nu-2)!$) and the aforementioned result of Wiles has been generalized to higher weight forms in \cite[Theorem 2.1]{FJ1995} (under the assumption that $p>\nu$). Hence one can virtually repeat the constructions in \cite[Sections 4 and 5]{LLZ2015} to obtain the stated result in the higher weight case.
\end{proof}
\end{theorem}


Now let $\psi_1$ (resp. $\psi_2$) be an algebraic Hecke character of $K$ of infinity type $(1-l,0)$ (resp. $(1-m,0)$) and conductor $(c)$ for some $c\in\Z_{\geq 1}$ coprime to $p$ and divisible only by primes split in $K$. We keep the notation introduced above adding a subscript $i\in\{1,2\}$. In this situation, we write $M:=N_{\psi_1}=N_{\psi_2}=c^2\cdot D_K$. Since the role played by $\psi_1$ and $\psi_2$ will be symmetric, it is harmless to assume directly that $l\geq m\geq 2$. In what follows we will always assume that $p\geq\max\{5,l+1\}$.

\begin{nota}
Let $n\in\Z_{\geq 1}$ be an integer coprime to $M$ and let $\bb{T}'_\nu(n)$ denote the $\Z_p$-subalgebra of $\End_{\Z_p}(H^1(Y(1,Mn)_{\bar{\Q}},\Ls_{\nu-2}(1))$ generated by all Hecke operators $T_q'$ and diamond operators $\langle d\rangle'$. We also denote by $\Delta_n$ the subgroup of diamond operators given by $\langle d\rangle '$ for integers $d\equiv 1\mod M$. 
There is an obvious diagonal morphism $\Z_p[\Delta_n]\to \bb{T}'_\nu(n)\otimes_{\Z_p}\bb{T}'_\mu(n)$ and we let $\bb{T}'_{l,m}(n)$ be the quotient of $\bb{T}'_\nu(n)\otimes_{\Z_p}\bb{T}'_\mu(n)$ defined as
\[
\bb{T}'_{l,m}(n)=\big(\bb{T}'_\nu(n)\otimes_{\Z_p}\bb{T}'_\mu(n)\big)\otimes_{\Z_p[\Delta_n]}\Z_p
\]
where the map $\Z_p[\Delta_n]\to\Z_p$ is the augmentation map.
\end{nota}

\begin{rmk}
The algebra $\bb{T}'_{l,m}(n)$ acts naturally on the $\Delta_n$-coinvariants
\[
\big(H^1_{\et}(Y(Mn)_{\bar{\Q}},\Ls_{l-2}(1))\otimes_{\Z_p}H^1_{\et}(Y(Mn)_{\bar{\Q}},\Ls_{m-2}(1))\big)_{\Delta_n}
\]
for the diagonal action of $\Delta_n$.
\end{rmk}
\begin{ass}\label{ass:selfdual}
We assume that $\chi_{\psi_1}\chi_{\psi_2}$ is the trivial character modulo $M$.
\end{ass}

\begin{nota}
For every prime $q$ that splits in $K$, we choose one of the primes $\frk{q}$ of $K$ dividing it and we write $\overline{\frk{q}}$ to denote the other prime ideal of $K$ above $q$.
For $n\in\Z_{\geq 1}$ coprime to $M$ we write $n=n^+n^-$, with $n^+$ divisible only by primes splitting in $K$ and $n^-$ squarefree and divisible only by primes inert in $K$. We write $n\cl{O}_K=\frk{n}^+\overline{\frk{n}}^+(n^-)$ as an ideal of $\cl{O}_K$ according to the choice made above, with $\frk{n}^+$ and $\overline{\frk{n}}^+$ ideals of norm $n^+$ such that $n^+\cl{O}_K=\frk{n}^+\overline{\frk{n}}^+$. Similarly, we write $c\,\cl{O}_K=\frk{c}\overline{\frk{c}}$.
\end{nota}

Since we assume that $p\geq 5$ does not divide the class number of $K$, it is easy to see that we have canonical identifications:
\[
\tau_{cn}: H(\frk{c}\frk{n}^+)\times H(\overline{\frk{c}}\overline{\frk{n}}^+)\times H(n^-)\xrightarrow{\cong} H(cn).
\]
We let
\[
\sigma_{cn}:H(\frk{c}\frk{n}^+(n^-))\times H(\overline{\frk{c}}\overline{\frk{n}}^+(n^-))\to H[cn]
\]
be the composition given as follows:
\begin{center}
    \begin{tikzcd}
       & H(\frk{c}\frk{n}^+(n^-))\times H(\overline{\frk{c}}\overline{\frk{n}}^+(n^-))\cong H(\frk{c}\frk{n}^+)\times H(n^-)\times H(\overline{\frk{c}}\overline{\frk{n}}^+)\times H(n^-)\arrow[d, "(\dagger)"]\\
       & H(cn)\cong H(cn^+)\times H(n^-)\arrow[d, twoheadrightarrow]\\
       &H[cn]
    \end{tikzcd}

\end{center}
where $(\dagger)$ sends a $4$-tuple $(x,a,y,b)$ to $(\tau_{cn^+}(x,y),(ab)^{1/2})$ (note that since we assume $p\geq 5$ this is well-defined) and the second arrow is the canonical projection.

We also let
\[
\sigma^{\mathbf{c}}_{cn}:H(\frk{c}\frk{n}^+(n^-))\times H(\frk{c}\frk{n}^+(n^-))\to H[cn]
\]
be given by $\sigma_{cn}\circ (1,\mathbf{c})$ where $\mathbf{c}:H(\frk{c}\frk{n}^+(n^-))\to H(\overline{\frk{c}}\overline{\frk{n}}^+(n^-))$ is given by the action of complex conjugation.

\begin{nota}
For $n\in\Z_{\geq 1}$ as above, we write
\[
\tilde{n}:=n^+\cdot (n^-)^2=N_{K/\Q}(\frk{n}^+(n^-))=N_{K/\Q}(\overline{\frk{n}}^+(n^-))
\]
and
\[
\frk{m}:=\frk{c}\frk{n}^+(n^-)\qquad \overline{\frk{m}}:=\overline{\frk{c}}\overline{\frk{n}}^+(n^-).
\]     
\end{nota}

\begin{lemma}\label{lemma:hecketorcg}
There is a unique morphism $\underline{\phi}_n:\bb{T}'_{l,m}(\tilde{n})\to\cl{O}[H[cn]] $ of $\Z_p$-algebras making the following diagram commutative
\begin{center}
\begin{tikzcd}[column sep = large]
& \bb{T}'_l(\tilde{n})\otimes_{\Z_p}\bb{T}'_m(\tilde{n})\arrow[r, "\phi_{1,\frk{n}}\otimes\phi_{2,\overline{\frk{n}}}"]\arrow[d,twoheadrightarrow] & \cl{O}[H(\frk{m})]\otimes_{\cl{O}} \cl{O}[H(\overline{\frk{m}})]\arrow[r,"\sigma_{cn}"] &\cl{O}[H[cn]]\\
& \bb{T}'_{l,m}(\tilde{n})\arrow[urr, bend right = 10, dashed, "\underline{\phi}_n"]
\end{tikzcd}
\end{center}
Similarly, there is a unique morphism $\underline{\phi}^{\mathbf{c}}_n:\bb{T}'_{l,m}(\tilde{n})\to\cl{O}[H[cn]]$ such the analogue of the above diagram with $\phi_{2,\overline{\frk{n}}}$ replaced by $\phi_{2,\frk{n}}$ and $\sigma_{cn}$ replaced by $\sigma^{\mathbf{c}}_{cn}$ is commutative.
\begin{proof}
From the construction and Assumption \ref{ass:selfdual}, it follows that
\[
\sigma_{cn}\circ (\phi_{1,\mathfrak{n}}\otimes\phi_{2,\mathfrak{\overline{n}}}) (\langle d\rangle'\otimes \langle d \rangle ')=1
\]
for every $\langle d\rangle'\in\Delta_n$. The required factorizations follow immediately.
\end{proof}
\end{lemma}



\begin{nota}
We fix the following notations:
\[
\begin{aligned}
& H^1(\psi_1,\frk{m}):=H^1_{\et}(Y_1(M\tilde{n})_{\overline{\bb{Q}}},\Ls_{l-2}(1))\otimes_{(\bb{T}_l'(M\tilde{n}),\phi_{1,\frk{m}})} \cl{O}[H(\frk{m})]\quad\text{and similarly }H^1(\psi_1,\overline{\frk{m}});\\
& H^1(\psi_2,\frk{m}):=H^1_{\et}(Y_1(M\tilde{n})_{\overline{\bb{Q}}},\Ls_{m-2}(1))\otimes_{(\bb{T}_m'(M\tilde{n}),\phi_{2,\frk{m}})} \cl{O}[H(\frk{m})]\quad\text{and similarly }H^1(\psi_2,\overline{\frk{m}});\\
& H^1(\psi_1,\psi_2,n):=\\
&\big(H^1_{\et}(Y(M\tilde{n})_{\bar{\Q}},\Ls_{l-2}(1))\otimes_{\Z_p}H^1_{\et}(Y(M\tilde{n})_{\bar{\Q}},\Ls_{m-2}(1))\big)_{\Delta_n} \otimes_{(\bb{T}'_{l,m}(\tilde{n}),\underline{\phi}_n)}\cl{O}[H[cn]]\otimes_{\cl{O}}\cl{O}(\tfrac{2-l-m}{2});\\
& H^1(\psi_1,\psi_2^{\mathbf{c}},n):=\\
&\big(H^1_{\et}(Y(M\tilde{n})_{\bar{\Q}},\Ls_{l-2}(1))\otimes_{\Z_p}H^1_{\et}(Y(M\tilde{n})_{\bar{\Q}},\Ls_{m-2}(1))\big)_{\Delta_n} \otimes_{(\bb{T}'_{l,m}(\tilde{n}),\underline{\phi}_n^{\mathbf{c}})}\cl{O}[H[cn]]\otimes_{\cl{O}}\cl{O}(\tfrac{2-l-m}{2}).
\end{aligned}
\]
\end{nota}

\begin{rmk}
\label{rmk: decompGalois}
Observe that as $G_\Q$-representations
\[
\Ind_{K}^\bb{Q}\cl{O}(\hat{\psi}_1^{-1})\otimes_{\cl{O}}\Ind_{K}^\bb{Q}\cl{O}(\hat{\psi}_2^{-1})\cong \Ind_{K}^\bb{Q}\cl{O}(\widehat{\psi_1\psi_2}^{-1})\oplus \Ind_{K}^\bb{Q}\cl{O}(\widehat{\psi_1\psi_2^c}^{-1})
\]
Note moreover that the $p$-adic Galois characters attached to (the $p$-adic avatars of) $\eta_1:=\psi_1\psi_2\vert\cdot\vert_{\bb{A}_K^\times}^{(l+m-2)/2}$ and $\eta_2:=\psi_1\psi_2^c\vert\cdot\vert_{\bb{A}_K^\times}^{(l+m-2)/2}$ factor through $H[cp^\infty]$ (actually, if $l=m$, $\eta_2$ is even a ring class character, factoring through $H[c]$). Hence we get a natural projection (compatible with the morphism $\sigma_{cn}$)
\begin{equation}
\label{projdecomp12}
\Ind_{K(\frk{m})}^\bb{Q}\cl{O}(\hat{\psi}_1^{-1})\otimes_{\cl{O}}\Ind_{K(\overline{\frk{m}})}^\bb{Q}\cl{O}(\hat{\psi}_2^{-1})\otimes_{\cl{O}}\cl{O}(\tfrac{2-l-m}{2})\twoheadrightarrow \Ind_{K[cn]}^\bb{Q}\cl{O}(\hat{\eta}_1^{-1}),
\end{equation}
as well as a natural projection (compatible with the morphism $\sigma^{\mathbf{c}}_{cn}$)
\begin{equation}
\label{projdecomp12c}
\Ind_{K(\frk{m})}^\bb{Q}\cl{O}(\hat{\psi}_1^{-1})\otimes_{\cl{O}}\Ind_{K(\frk{m})}^\bb{Q}\cl{O}(\hat{\psi}_2^{-1})\otimes_{\cl{O}}\cl{O}(\tfrac{2-l-m}{2})\twoheadrightarrow \Ind_{K[cn]}^\bb{Q}\cl{O}(\hat{\eta}_2^{-1}).
\end{equation}
\end{rmk}

\begin{corollary}\label{cor:LLZ}
Let $\cl{B}_M$ be the set of positive integers coprime to $pM$. Then for every $n\in\cl{B}$ there is a unique $G_\bb{Q}$-equivariant surjective morphisms of $\cl{O}[H[cn]]$-modules 
\[
\underline{\nu}_n: H^1(\psi_1,\psi_2,n)\twoheadrightarrow \Ind_{K[cn]}^\bb{Q}\cl{O}(\hat{\eta}_1^{-1})
\]
making the following diagram commutative.
\begin{center}
\begin{tikzcd}
&\big(H^1(\psi_1,\frk{m})\otimes_{\cl{O}}H^1(\psi_2,\overline{\frk{m}})\big)(\tfrac{2-l-m}{2})\arrow[r,  "\nu_1\otimes\nu_2", "\cong"']\arrow[d] &\big(\Ind_{K(\frk{m})}^\bb{Q}\cl{O}(\hat{\psi}_1^{-1})\otimes_{\cl{O}}\Ind_{K(\overline{\frk{m}})}^\bb{Q}\cl{O}(\hat{\psi}_2^{-1})\big)(\tfrac{2-l-m}{2})\arrow[d,twoheadrightarrow, "\eqref{projdecomp12}"]\\
&H^1(\psi_1,\psi_2,n)\arrow[r, dashed, twoheadrightarrow, "\underline{\nu}_{n}"] &\Ind_{K[cn]}^\bb{Q}\cl{O}(\hat{\eta}_1^{-1})
\end{tikzcd}
\end{center}
The left vertical arrow in the above diagram is induced by the natural quotient map.

Moreover, whenever $n\mid n'$ for elements of $\cl{B}_M$, the diagram
\begin{center}
\begin{tikzcd}
& H^1(\psi_1,\psi_2,n')\arrow[r, twoheadrightarrow, "\underline{\nu}_{n'}"]\arrow[d, "\cl{N}_n^{n'}"] &\Ind_{K[cn']}^\bb{Q}\cl{O}(\hat{\eta}_1^{-1}) \arrow[d, "\mathrm{Norm}^{K[cn']}_{K[cn]}"]\\
&H^1(\psi_1,\psi_2,n)\arrow[r, twoheadrightarrow, "\underline{\nu}_n"] &\Ind_{K[cn]}^\bb{Q}\cl{O}(\hat{\eta}_1^{-1})
\end{tikzcd}
\end{center}
commutes, where $\cl{N}_{n}^{n'}$ is induced by  $\big(\cl{N}_{\frk{m}}^{\frk{m}'}\otimes 1 \big)\otimes\big(\cl{N}_{\overline{\frk{m}}}^{\overline{\frk{m}'}}\otimes 1 \big)$.

Similarly, we get unique $G_\bb{Q}$-equivariant surjective morphisms of $\cl{O}[H[cn]]$-modules 
\[
\underline{\nu}^\mathbf{c}_n: H^1(\psi_1,\psi^\mathbf{c}_2,n)\twoheadrightarrow \Ind_{K[cn]}^\bb{Q}\cl{O}(\hat{\eta}_2^{-1})
\]
which make the analogue diagrams commutative, i.e.,
\begin{center}
\begin{tikzcd}
&\big(H^1(\psi_1,\frk{m})\otimes_{\cl{O}}H^1(\psi_2,\frk{m})\big)(-1)\arrow[r,  "\nu_1\otimes\nu^\mathbf{c}_2", "\cong"']\arrow[d] &\big(\Ind_{K(\frk{m})}^\bb{Q}\cl{O}(\hat{\psi}_1^{-1})\otimes_{\cl{O}}\Ind_{K(\frk{m})}^\bb{Q}\cl{O}(\hat{\psi}_2^{-1})\big)(-1)\arrow[d,twoheadrightarrow, "\eqref{projdecomp12c}"]\\
&H^1(\psi_1,\psi^\mathbf{c}_2,n)\arrow[r, dashed, twoheadrightarrow, "\underline{\nu}^\mathbf{c}_{n}"] &\Ind_{K[cn]}^\bb{Q}\cl{O}(\hat{\eta}_2^{-1})
\end{tikzcd}
\end{center}
and, whenever $n\mid n'$ for elements of $\cl{B}_M$,
\begin{center}
\begin{tikzcd}
& H^1(\psi_1,\psi^\mathbf{c}_2,n')\arrow[r, twoheadrightarrow, "\underline{\nu}^\mathbf{c}_{n'}"]\arrow[d, "^\mathbf{c}\cl{N}_n^{n'}"] &\Ind_{K[cn']}^\bb{Q}\cl{O}(\hat{\eta}_2^{-1}) \arrow[d, "\mathrm{Norm}^{K[cn']}_{K[cn]}"]\\
&H^1(\psi_1,\psi_2^\mathbf{c},n)\arrow[r, twoheadrightarrow, "\underline{\nu}^\mathbf{c}_n"] &\Ind_{K[cn]}^\bb{Q}\cl{O}(\hat{\eta}_2^{-1})
\end{tikzcd}
\end{center}
with $^\mathbf{c}\cl{N}_n^{n'}$ induced by  $\big(\cl{N}_{\frk{m}}^{\frk{m}'}\otimes 1\big)^{\otimes 2}$.
\end{corollary}
\begin{proof}
It is easy to see that the required factorization follows from the construction. The commutativity of the second diagram (in each case) follows directly now from Theorem~\ref{thm:LLZ}, the compatibility between corestriction morphisms and the uniqueness property of $\underline{\nu}_n$.
\end{proof}

\subsection{Improved diagonal classes}
\label{subsec: improved diagonal classes}
Let $\mathbf{r}=(r_1,r_2,r_3)\in(\Z_{\geq 0})^3$ be a triple of non-negative integers which is balanced, i.e., $r_a<r_b+r_c$ for any permutation $\{a,b,c\}$ of $\{1,2,3\}$. 

\medskip
Following \cite[Section 3]{BSVast1}, one can construct diagonal classes 
\[
\kappa_{N(n),\mathbf{r}}\in H^1\big(\Q, H^3_{\et}(Y^3_{\bar{\Q}},\Ls_{[\mathbf{r}]}(2-r))\big)\]
for some $N\geq 5$, $n\geq 1$, where $Y=Y(1,N(n))_\Q$.

\medskip
Here is a diagram depicting the situation. 

\begin{center}
    \begin{tikzcd}
   & \DET_{N(n),\mathbf{r}}^{\et}\in H^0_{\et}(Y,\Ss_\mathbf{r}(r))\arrow[d, bend right = 90, mapsto]\arrow[r, "d_{*}"]\arrow[dr, "AJ_{\et}"]& H^4_{\et}(Y^3,\Ss_{[\mathbf{r}]}(r+2))
    \arrow[d, "\mathrm{HS}"] \\
    & \kappa_{N(n),\mathbf{r}}\ni H^1\big(\Q, H^3_{\et}(Y^3_{\bar{\Q}},\Ls_{[\mathbf{r}]}(2-r))_{\Q_p}\big)
    & H^1\big(\Q, H^3_{\et}(Y^3_{\bar{\Q}},\Ss_{[\mathbf{r}]}(r+2))_{\Q_p}\big)\arrow[l, "s_\mathbf{r}", "\cong"']
\end{tikzcd}
\end{center}

\begin{rmk}
\label{AJetrmk}
The following discussion explains the diagram above.
\begin{enumerate}[(i)]
    \item $d: Y\hookrightarrow Y\times Y\times Y$ is the diagonal embedding and $d_{*}$ is the corresponding Gysin map.
    \item For $\mathscr{F}\in\{\Ss,\Ls\}$ and $\mathbf{r}=(r_1,r_2,r_3)$ as above, the $\Z_p$-local system $\mathscr{F}_\mathbf{r}$ on $Y$ is defined as
    \[    \mathscr{F}_\mathbf{r}:=\mathscr{F}_{r_1}\otimes_{\Z_p}\mathscr{F}_{r_2}\otimes_{\Z_p}\mathscr{F}_{r_3}
    \]
    and the $\Z_p$-local system $\mathscr{F}_{[\mathbf{r}]}$ on $Y^3$ is defined as
    \[        \mathscr{F}_{[\mathbf{r}]}:=p_1^*\mathscr{F}_{r_1}\otimes_{\Z_p}p_2^*\mathscr{F}_{r_2}\otimes_{\Z_p} p_3^*\mathscr{F}_{r_3}\,,
    \]
    where $p_j: Y\times Y\times Y\to Y$ is the natural projection on the $j$-th factor. In particular, it follows that $d^*\mathscr{F}_{[\mathbf{r}]}=\mathscr{F}_\mathbf{r}$.
    \item The map $\mathrm{HS}$ is a morphism coming from the Hochschild-Serre spectral sequence
    \begin{equation}
    \label{HSspectral}
        H^i(\Q,H^j_{\et}(Y_{\bar{\Q}}^3,\Ss_{[\mathbf{r}]}(r+2)))\Rightarrow H^{i+j}_{\et}(Y^3, \Ss_{[\mathbf{r}]}(r+2))\,.
    \end{equation}
    The existence of such a morphism is due to the fact that $H^4_{\et}(Y_{\bar{\Q}}^3,\Ss_{[\mathbf{r}]}(r+2))=0$ (by Artin vanishing theorem, as $Y$ is an affine curve).
    \item The \'{e}tale Abel-Jacobi map in this setting is defined as $AJ_{\et}:=d_{*}\circ\mathrm{HS}$.
    \item The isomorphism $s_\mathbf{r}$ is induced by the isomorphisms $s_{r_j}$ for $j\in\{1,2,3\}$ (cf. \eqref{isosr} above).
\end{enumerate}
\end{rmk}

\medskip
The class $\kappa_{N(n),\mathbf{r}}$ is the image under the composition $s_{\mathbf{r}}\circ AJ_{\et}$ of an element
    \begin{equation}
    \label{detrdef}
        \DET_{N(n),\mathbf{r}}^{\et}\in H^0_{\et}(Y,\Ss_\mathbf{r}(r)),
    \end{equation}
constructed using classical invariant theory (cf. the discussion in \cite[pp. 94-95]{BSVast1}). When $n=1$, we simply write $\kappa_{N,\mathbf{r}}=\kappa_{N(1),\mathbf{r}}$.

\medskip
Fix now integers $N\geq 5$ and $n\geq 1$ and let $W_n:=\big(Y_1(Nn)\times Y_1(Nn)\big)/\Delta_{n}$ (quotient taken with respect to the diagonal (and free) action of $\Delta_{n}$ on $Y_1(Nn)\times Y_1(Nn)$) and denote $q_n: Y_1(Nn)\times Y_1(Nn)\to W_n$ the quotient map. Note that there is a well-defined \emph{diagonal} closed immersion $\iota_n:Y_1(M(n))\hookrightarrow W_n$ with projections $q_{i,n}: W_n\to Y_1(M(n))$ for $i=2,3$.
Thanks to the isomorphisms \eqref{eq:iso-q}, for $\Fs\in\{\Ss,\Ls\}$ there is a natural isomorphism 
\[
\Fs_{1,Mn,r_2}\boxtimes \Fs_{1,Mn,r_3}:=p_{2,n}^*\Fs_{1,Mn,r_2}\otimes p_{3,n}^*\Fs_{1,Mn,r_3}\cong q_n^*(\Fs_{r_2,r_3}^{\Delta_n})
\]
where we write
\[
\Fs_{l-2,m-2}^{\Delta_n}:=q_{2,n}^*\Fs_{1,M(n),r_2}\otimes q_{3,n}^*\Fs_{1,M(n),r_3}.
\]
and $p_{i,n}:Y_1(Nn)^3\to Y_1(Nn)$ for $i\in\{1,2,3\}$ are the natural projections.

\medskip
Hence, there is an Hochschild-Serre spectral sequence (here $j\in\Z$ is any integer) 
\[
E_2^{p,q}=H^p\big(\Delta_n, H^q_{\etcc}(Y_1(Nn)_{\bar{\Q}}^2,\Ss_{r_2}\boxtimes\Ss_{r_3}(j))_{\Q_p}\big)\Rightarrow H^{p+q}_{\etcc}(W_{n,\bar{\Q}},\Ss_{r_2,r_3}^{\Delta_n}(j))_{\Q_p}
\]

Since $E_2^{p,q}=0$ for $q\leq 1$ (thanks to Artin vanishing and Poincar\'{e} duality), one can use this spectral sequence to show that pullback by $q_n$ induces isomorphisms
\[
q_n^*:H^2_{\etcc}(W_{n,\bar{\Q}},\Ss_{r_2,r_3}^{\Delta_n}(j))_{\Q_p}\xrightarrow{\cong} H^2_{\etcc}(Y_1(Nn)_{\bar{\Q}}^2,\Ss_{r_2}\boxtimes\Ss_{r_3}(j))^{\Delta_n}_{\Q_p}
\]
and by Poincar\'{e} duality we obtain isomorphisms 
\begin{equation}
\label{pushforwardiso}
q_{n,*}:H^2_{\et}(Y_1(Nn)_{\bar{\Q}}^2,\Ls_{l-2}\boxtimes\Ls_{m-2}(j))_{\Delta_n,\Q_p}\xrightarrow{\cong}H^2_{\et}(W_{n,\bar{\Q}},\Ls^{\Delta_n}_{l-2,m-2}(j))_{\Q_p}.    
\end{equation}

It is clear that one can apply the same machinery to the \emph{diagonal} embedding (obtained using $\iota_n$ on the second factor)
\[
d_{\Delta_n}:Y_1(N(n))\hookrightarrow Y_1(N(n))\times W_n
\]
to obtain classes
\[
\kappa_{N,\Delta_n,\mathbf{r}}\in H^1\big(\Q, H^3_{\et}(Z_{n,\bar{\Q}},\Ls^{\Delta_n}_{[\mathbf{r}]}(2-r))_{\Q_p}\big)
\]
where $Z_n= Y_1(N(n))\times W_n$ and, letting $p_{Y,n}:Z_n\to Y_1(N(n))$ and $p_{W,n}:Z_n\to W_n$ denote the obvious projections, 
\[
\Ls^{\Delta_n}_{[\mathbf{r}]}:=p_{Y,n}^*\Ls_{r_1}\otimes_{\Z_p}p_{W,n}^*\Ls^{\Delta_n}_{r_2,r_3}.
\]
We have a commutative diagram
\begin{center}
\begin{tikzcd}
    &Y_1(Nn) \arrow[r, "\mu_n"]\arrow[d, hookrightarrow, "d"] &Y_1(N(n)) \arrow[d, hookrightarrow, "d_{\Delta_n}"]\\
    &Y_1(Nn)^3 \arrow[r, "\text{$(\mu_n, q_n)$}"] &Z_n  
\end{tikzcd}
\end{center}
and, arguing in the same way as in the proof of \cite[Equation (174)]{BSVast1}, one has that 
\[
\mu_{n,*}(\DET_{Nn,\mathbf{r}})=\phi(n)\cdot \DET_{N(n),\mathbf{r}},    
\]
where $\phi(-)$ denotes Euler's $\phi$ function (note that $\phi(n)$ is the degree of the Galois cover $\mu_n$). It follows immediately (using the functoriality of Hochschild-Serre spectral sequences) that
\begin{equation}
\label{phinfactor}
(\mu_n,q_n)_*(\kappa_{Nn,\mathbf{r}})=\phi(n)\cdot \kappa_{N,\Delta_n,\mathbf{r}}.
\end{equation}
We can apply K\"{u}nneth decomposition (without having to invert $p$ thanks to our assumptions) and the isomorphisms \eqref{pushforwardiso}, starting from the classes $\kappa_{N,\Delta_n,\mathbf{r}}$ to produce classes
\[
\kappa^{(1)}_{n,\mathbf{r}}\in H^1\Big(\Q,H^1(Y_1(N(n))_{\bar{\Q}},\Ls_{r_1}(1))_{\Q_p}\otimes_{\Q_p}H_{n,r_2,r_3}(-1-r)\Big)
\]
where
\[
H_{n,r_2,r_3}=\big(H^1_{\et}(Y(Nn)_{\bar{\Q}},\Ls_{r_2}(1))_{\Q_p}\otimes_{\Q_p}H^1_{\et}(Y(Nn)_{\bar{\Q}},\Ls_{r_3}(1))_{\Q_p}\big)_{\Delta_n}
\]
\begin{rmk}
\label{phiinfactorrmk}
It follows directly from \eqref{phinfactor} that the image of the class $\kappa_{Nn,\mathbf{r}}$ in the above Galois cohomology group (obtained again applying K\"{u}nneth decomposition and projection on $\Delta_n$-coinvariants) equals $\phi(n)\cdot \kappa^{(1)}_{n,\mathbf{r}}$.    
\end{rmk}

We next consider the following modified diagonal classes. 
For each positive integer $n$ coprime to $p$ and $N$, let
\begin{equation*}
\kappa_{n,\mathbf{r}}^{(2)} := (\pi_2,1,1)_* \kappa_{n,\mathbf{r}}^{(1)}
\end{equation*}
lying in 
\[
H^1\Big(\Q,H^1(Y_1(N)_{\bar{\Q}},\Ls_{r_1}(1))_{\Q_p}\otimes_{\Q_p}H_{n,r_2,r_3}(-1-r)\Big)
\]
where $\pi_2:Y_1(N(n))\to Y_1(N)$ is defined as in \eqref{eq:pi12}.

\subsection{Proof of the tame norm relations}
\label{subsec:tameeulersystem}
We start recalling some results concerning the behaviour of diagonal classes under degeneracy maps.

\begin{lemma}\label{lemma:normrelations1}
Let $n$ be a positive integer and let $q$ be a prime number. Assume that both $n$ and $q$ are coprime to $N$. Write $r_q=q-1$ if $q\nmid n$ and $r_q=q$ if $q\mid n$. Then
\begin{align*}
&(\pr_q,\pr_1,\pr_1)_\ast \kappa_{Nnq,\mathbf{r}} = r_q(T_q,1,1)\kappa_{Nn,\mathbf{r}};\quad (\pr_1,\pr_q,\pr_q)_\ast \kappa_{Nnq,\mathbf{r}} = r_q \cdot q^{r-r_1}(T_q',1,1)\kappa_{Nn,\mathbf{r}}; \\
&(\pr_1,\pr_q,\pr_1)_\ast \kappa_{Nnq,\mathbf{r}} = r_q(1,T_q,1)\kappa_{Nn,\mathbf{r}};\quad (\pr_q,\pr_1,\pr_q)_\ast \kappa_{Nnq,\mathbf{r}} = r_q \cdot q^{r-r_2}(1,T_q',1)\kappa_{Nn,\mathbf{r}}; \\
&(\pr_1,\pr_1,\pr_q)_\ast \kappa_{Nnq,\mathbf{r}} = r_q (1,1,T_q)\kappa_{Nn,\mathbf{r}};\quad (\pr_q,\pr_q,\pr_1)_\ast \kappa_{Nnq,\mathbf{r}} =r_q \cdot q^{r-r_3}(1,1,T_q')\kappa_{Nn,\mathbf{r}}.
\end{align*}
If $q$ is coprime to $n$ we also have
\[
(\pr_1,\pr_1,\pr_1)_\ast \kappa_{Nnq,\mathbf{r}} = (q^2-1)\kappa_{Nn,\mathbf{r}};\quad (\pr_q,\pr_q,\pr_q)_\ast \kappa_{Nnq,\mathbf{r}} = (q^2-1)q^{r}\kappa_{Nn,\mathbf{r}};
\]
while if $q\mid n$ we have
\[
(\pr_1,\pr_1,\pr_1)_\ast \kappa_{Nnq,\mathbf{r}} = q^2\cdot\kappa_{Nn,\mathbf{r}};\quad (\pr_q,\pr_q,\pr_q)_\ast \kappa_{Nnq,\mathbf{r}} = q^{r+2}\kappa_{Nn,\mathbf{r}}.
\]
\begin{proof}
The same arguments proving equations (174) and (176) in \cite{BSVast1} yield these identities, adding the prime $q$ to the level rather than the prime $p$.
\end{proof}
\end{lemma}

\begin{lemma}\label{lemma:normrelations2}
Let $n$ be as above and let $q$ be a rational prime coprime to $nN$. Write $s_q=\phi(q^2)=q(q-1)$. Then
\begin{align*}
&(\pr_{qq},\pr_{11},\pr_{11})_*( \kappa_{Nnq^2,\mathbf{r}})=s_q\big((T^2_q,1,1)-(q+1)q^{r_1}(\langle q\rangle,1,1)\big)\kappa_{Nn,\mathbf{r}}; \\
&(\pr_{qq},\pr_{11},\pr_{qq})_*( \kappa_{Nnq^2,\mathbf{r}})=s_q\big(q^{2r-2r_2}(1,T'^{2}_q,1)-(q+1)q^{2r}(1,\langle q\rangle',1)\big)\kappa_{Nn,\mathbf{r}}; \\
&(\pr_{qq},\pr_{qq},\pr_{11})_*( \kappa_{Nnq^2,\mathbf{r}})=s_q\big(q^{2r-2r_3}(1,1,T'^{2}_q)-(q+1)q^{2r}(1,1,\langle q\rangle')\big)\kappa_{Nn,\mathbf{r}}; \\
&(\pr_{qq},\pr_{qq},\pr_{qq})_*( \kappa_{Nnq^2,\mathbf{r}})=(q^2-1)q^{2r+2}\kappa_{Nn,\mathbf{r}}.
\end{align*}
\begin{proof}
The equalities follow applying the equations of Lemma \ref{lemma:normrelations1} twice and using the standard relations between Hecke operators and pushforward via degeneracy maps.
\end{proof}
\end{lemma}

\begin{lemma}\label{lemma:idealFrob}
Let $n$ be as above and let $q$ be a rational prime coprime to $pnN$. If $q$ splits in $K$ as $(q)=\frk{q}\overline{\frk{q}}$, the composition of $\sigma_{cn}:H(\frk{c}\frk{n}^+(n^-))\times H(\overline{\frk{c}}\overline{\frk{n}}^+(n^-))\to H[cn]$ with the (inverse of the) Artin map from class field theory (normalized geometrically) sends:
\begin{align*}
([(q)],[(q)])\mapsto 1,\quad &\quad ([\frk{q}],[\frk{q}])\mapsto \Frob_{\frk{q}}^{-1},\\
([\overline{\frk{q}}],[\overline{\frk{q}}])\mapsto \Frob_{\frk{q}},\quad&\quad ([\frk{q}],[\overline{\frk{q}}])\mapsto 1.
\end{align*}
Similarly $\sigma^\mathbf{c}_{cn}$ sends:
\begin{align*}
([(q)],[(q)])\mapsto 1,\quad &\quad ([\frk{q}],[\overline{\frk{q}}])\mapsto \Frob_{\frk{q}}^{-1},\\
([\overline{\frk{q}}],[\frk{q}])\mapsto \Frob_{\frk{q}},\quad&\quad ([\frk{q}],[\frk{q}])\mapsto 1.
\end{align*}
If $q$ is inert in $K$, we simply have $([(q)],[(q)])\mapsto 1$ for both $\sigma_{cn}$ and $\sigma^\mathbf{c}_{cn}$.
\begin{proof}
This is an easy computation following directly from the definition of $\sigma_{cn}$ and $\sigma^\mathbf{c}_{cn}$ and recalling that in $H[cn]$ we have $\Frob_{\frk{q}}^{-1}\Frob_{\overline{\frk{q}}}^{-1}=1$.
\end{proof}
\end{lemma}

Now we fix $f\in S_k(\Gamma_0(N_f))$ a newform (for some even integer $k\geq 2$) and assume that $g=\theta_{\psi_1}$ and $h=\theta_{\psi_2}$ (for $\psi_1$ and $\psi_2$ Hecke characters as in Section \ref{subsec:llz}). Note that the assumption that $\chi_g\chi_h=1$ implies that $l\equiv m\mod 2$ for the weights of $g$ and $h$. Recall that we assume $l\geq m$.

The level of $g$ and $h$ will be again denoted by $M$ and we assume that $(N_f,M)=1$. We always let $E$ denote a finite extension of $\Q_p$ containing all the relevant coefficients (with ring of integers $\cl{O}$).

Set $\mathbf{r}:=(r_1, r_2, r_3):=(k-2,l-2,m-2)$ and we assume that it is a balanced triple.

\medskip
Let $N=N_f\cdot M$, and (since $N$ will be fixed in what follows) put $Y(n)= Y_1(Nn)$ for every positive integer $n$ (with the convention $Y=Y(1)$). We will assume that $p\nmid N$ unless otherwise stated. 

\medskip
Fix test vectors
\[
\breve{f}\in S_k(\Gamma_0(N))[f],\quad  \breve{g}\in S_l(N,\chi_g)[g],\quad \breve{h}\in S_m(N,\chi_h)[h].
\]
These test vectors determine maps
\begin{align*}
H^1_{\et}(Y_{\overline{\Q}},\Ls_{k-2}(1))_{\cl{O}} &\rightarrow H^1_{\et}(Y_1(N_f)_{\overline{\Q}},\Ls_{k-2}(1))_{\cl{O}} \\
H^1_{\et}(Y(n)_{\overline{\Q}},\Ls_{l-2}(1))_{\cl{O}} &\rightarrow H^1_{\et}(Y_1(Mn)_{\overline{\Q}},\Ls_{l-2}(1))_{\cl{O}} \\
H^1_{\et}(Y(n)_{\overline{\Q}},\Ls_{m-2}(1))_{\cl{O}} &\rightarrow H^1_{\et}(Y_1(Mn)_{\overline{\Q}},\Ls_{m-2}(1)))_{\cl{O}}
\end{align*}

Let $n$ be a squarefree positive integer coprime to $pN$ and recall the notation $\tilde{n}$ of the previous Section \ref{subsec:llz}.

We use the above maps to project the classes $\kappa_{\tilde{n},\mathbf{r}}^{(2)}$ to classes 
\begin{equation*}
\tilde{\kappa}_{n,f,\underline{\psi}}\in H^1\big(\Q, V(f)(1-k/2)\otimes_E H^1(\psi_1,\psi_2,n)_E\big),
\end{equation*}
where $\underline{\psi}=(\psi_1,\psi_2)$.

Finally, via the morphisms $\underline{\nu}_n$ of Corollary \ref{cor:LLZ}, Shapiro's lemma and corestriction from $K[cn]$ to $K[n]$ we obtain classes
\begin{equation}
\tilde{\kappa}_{n,f,\eta_1}\in H^1(K[n], V(f)(1-k/2)\otimes_{E}\hat{\eta}_1^{-1}),    
\end{equation}
where recall that $\eta_1:=\psi_1\psi_2\vert\cdot\vert_{\bb{A}_K^\times}^{(l+m-2)/2}$.

The same procedure using $\underline{\nu}^\mathbf{c}_n$ produces classes
\begin{equation}
\tilde{\kappa}_{n,f,\eta_2}\in H^1(K[n], V(f)(1-k/2)\otimes_{E}\hat{\eta}_2^{-1}),    
\end{equation}
where recall that $\eta_2:=\psi_1\psi_2^c\vert\cdot\vert_{\bb{A}_K^\times}^{(l+m-2)/2}$.

\begin{proposition}\label{prop:normrelations3}
Let $n$ be a positive integer coprime to $p\cdot (l-2)!\cdot N$, and let $q$ be a rational prime coprime to $p\cdot n\cdot N\cdot (l-2)!$.
\begin{enumerate}[(i)]
\item If $q$ splits in $K$ as $(q) = \frk{q}\overline{\frk{q}}$, then
\begin{align*}
&\cor_{K[nq]/K[n]}(\tilde{\kappa}_{nq,f,\eta_1})=\\
&\Bigg(a_q(f)-q^{k/2-1}\cdot \eta_1(\frk{q})\Frob_{\frk{q}}^{-1}-q^{k/2-1}\cdot\eta_1(\overline{\frk{q}})\Frob_{\frk{q}}+\frac{q^r(1-q)}{q^{l+m-2}}\cdot \psi_1(\frk{q})\psi_2(\overline{\frk{q}})\Bigg)(\tilde{\kappa}_{n,f,\eta_1})
\end{align*}
and
\begin{align*}
&\cor_{K[nq]/K[n]}(\tilde{\kappa}_{nq,f,\eta_2})=\\
&\Bigg(a_q(f)-q^{k/2-1}\cdot\eta_2(\frk{q})\Frob_{\frk{q}}^{-1}-q^{k/2-1}\cdot\eta_2(\overline{\frk{q}})\Frob_{\frk{q}}+\frac{q^r(1-q)}{q^{l+m-2}}\cdot \psi_1(\frk{q})\psi_2(\frk{q})\Bigg)(\tilde{\kappa}_{n,f,\eta_2}).
\end{align*}
\item If $q$ is inert in $K$, then for $i\in\{1,2\}$ we have
\[
\cor_{K[nq]/K[n]}(\tilde{\kappa}_{nq,f,\eta_i})=\bigg(a_q(f)^2-\frac{(q+1)}{q}\cdot q^{k-2}(q-1+q^{l-2}+q^{m-2})\bigg)(\tilde{\kappa}_{n,f,\eta_i}).
\]
\end{enumerate}
\begin{proof}
The commutativity of the two diagrams in the statement of corollary \ref{cor:LLZ} allows us to compute
\[
\cor_{K[nq]/K[n]}(\tilde{\kappa}_{nq,f,\eta_1})=\big(\cor_{K[cn]/K[n]}\circ\,\underline{\nu}_n\circ(\pr_{\tilde{n}},1,1)_*\circ (\pr_{\tilde{q},*},\cl{N}_n^{nq})\big)(\tilde{\kappa}_{n,f,\eta_1}).
\]
If $q$ splits, we have $\tilde{q}=q$ and $\cl{N}_n^{nq}=\big(\pr_{1}\otimes 1-\pr_{q}\otimes A\,,\,\pr_{1}\otimes 1-\pr_{q}\otimes B\big)$,
where
\[
A=\frac{\psi_1(\frk{q})[\frk{q}]}{q^{l-1}},\qquad B=\frac{\psi_2(\overline{\frk{q}})[\overline{\frk{q}}]}{q^{m-1}}.
\]
The stated formula follows from an elementary computation where one needs to apply the results of Proposition \ref{prop:LLZ}, Lemma \ref{lemma:normrelations2} and Lemma \ref{lemma:idealFrob}.

If $q$ is inert, we have $\tilde{q}=q^2$ and $\cl{N}_n^{nq}=\big(\pr_{11}\otimes 1-\pr_{qq}\otimes C\,,\,\pr_{11}\otimes 1-\pr_{qq}\otimes D\big)$, where
\[
C=\frac{\psi_1((q))[(q)]}{q^{2l-2}},\qquad D=\frac{\psi_2((q))[(q)]}{q^{2m-2}}.
\]
The stated formula follows again by the aforementioned results. 
Note that thanks to our construction (cf. Remark \ref{phiinfactorrmk}), we avoid the (disturbing) factor $\phi(\tilde{q})$ in front of the norm relations. Such factor would appear if we had not modified the diagonal classes suitably.

The formulas for $\tilde{\kappa}_{n,f,\eta_2}$ follow from analogous computations. 
\end{proof}
\end{proposition}

In particular, restricting to positive integers $n$ as above that are divisible only by primes $q$ splitting in $K$, this leads to the following result.

\begin{theorem}\label{principal:tame}
Let $\mathcal{S}$ be the set of squarefree products of primes $q$ which split in $K$ and are coprime to $p\cdot N\cdot (l-2)!$. Assume that $p\geq\max\{5,l+1\}$ and that $H^1(K[n],T(f))$ is $\cl{O}$-free for all $n\in\mathcal{S}$. Then there exists two collections of classes
\[
\big\{\kappa_{n,f,\eta_1}\in H^1(K[n],T(f)(1-k/2)\otimes\hat{\eta}_1^{-1})\mid n\in\mathcal{S}\big\}\qquad \big\{\kappa_{n,f,\eta_2}\in H^1(K[n],T(f)\otimes\hat{\eta}_2^{-1})\mid n\in\mathcal{S}\big\}
\]
such that whenever $n, nq\in\mathcal{S}$ with $q$ a prime, we have for $i\in\{1,2\}$
\begin{equation}\label{eq:norm-split}
{\rm cor}_{K[nq]/K[n]}(\kappa_{nq,f,\eta_i})=P_{i,\frk{q}}({\rm Fr}_{\frk{q}}^{-1})\kappa_{nq,f,\eta_i},
\end{equation}
where $\mathfrak{q}$ is any of the primes of $K$ above $q$, and $P_{i,\frk{q}}(X)=\det\big(1-X\cdot {\Frob}_{\frk{q}}^{-1}\vert (V^*(f)(1)\otimes\hat{\eta}_i)\big)$.
Moreover, $\kappa_{n,f,\eta_1}=\tilde{\kappa}_{n,f\underline{\psi}}$ and $\kappa_{n,f,\eta_2}=\tilde{\kappa}^\mathbf{c}_{n,f\underline{\psi}}$ whenever $K[n]/K$ is unramified outside $pN$.
\end{theorem}

\begin{proof}
We prove the result for $i=1$ (the proof for $i=2$ is identical). We start observing that the possible denominators for the classes $\tilde{\kappa}_{n,f,\eta_1}$ are divisors of $(l-2)!$ (recall that we are assuming that $l\geq m$), as it follows from the discussion after Equation \eqref{isosr}. Since we are forced to assume that $p>l$ by Theorem \ref{thm:LLZ}, it follows that under our hypothesis the classes $\tilde{\kappa}_{n,f,\eta_1}$ can be lifted uniquely to classes with coefficients in $T(f)(1-k/2)\otimes_{\cl{O}}\hat{\eta_1}^{-1}$. 
One easily computes that
\[
P_{1,\frk{q}}(X)=1-\frac{a_q(f)\eta_1(\frk{q})}{q^{k/2}}X+\frac{\eta_1(\frk{q})^2}{q}X^2\,.
\]
We set
\[
Q_{1,\frk{q}}(X)=a_q(f)+\frac{q^r(1-q)}{q^{l+m-2}}\psi_1(\frk{q})\psi_2(\overline{\frk{q}})-q^{k/2-1}
\cdot\eta_1(\frk{q})X^{-1}-q^{k/2-1}\cdot\eta_1(\overline{\frk{q}})X\,.
\]
Applying \cite[Lemma 9.6.3]{Rub2000} (with $f_\frk{q}=Q_{1,\frk{q}}$, $u_\frk{q}=-\eta_1(\frk{q})/q^{k/2}$ and $d=1$ in the notation of \cite{Rub2000}) to the previous Proposition \ref{prop:normrelations3}, one obtains a collection of classes satisfying norm relations for the polynomial
\[
\tilde{P}_{1,\frk{q}}(X)=P_{1,\frk{q}}(X)+\frac{1-q}{q}+\frac{q-1}{q^{(l+m+2)/2}}\psi_1(\frk{q})\psi_2(\overline{\frk{q}})\eta_1(\frk{q})X\,.
\]
Since (visibly) $\tilde{P}_{1,\frk{q}}(X)\equiv P_{1,\frk{q}}(X)\mod (q-1)$, we can apply \cite[Lemma 9.6.1]{Rub2000} to obtain a collection of classes $\big\{\kappa_{n,f,\eta_1}\in H^1(K[n],T(f)\otimes\hat{\eta}_1^{-1})\mid n\in\mathcal{S}\big\}$ that satisfies the stated norm relations and properties.
\end{proof}

\begin{rmk}
In the inert case, writing $\frk{q}=(q)$ we have
\[
P_{1,\frk{q}}(X)=\det\big(1-X\cdot {\Frob}_{\frk{q}}^{-1}\vert (V^*(f)(1)\otimes\hat{\eta}_1)\big)=1-\frac{a_q(f)^2}{q^k}X+\frac{2X}{q}+\frac{X^2}{q^2}
\]
and, as in the proof of Theorem~\ref{principal:tame}, we find the congruence (recall that $k$ and $l+m$ are even integers)
\[
-P_{1,\frk{q}}(\Frob_{\frk{q}}^{-1})\equiv a_q(f)^2-\frac{(q+1)^2}{q}\equiv a_q(f)^2-\frac{(q+1)}{q}\cdot q^{k-2}(q-1+q^{l-2}+q^{m-2}) \mod (q^2-1)
\]
as endomorphisms of $H^1(K[n],V^*(f)(1)\otimes\hat{\eta}_1)$ (and similarly for $P_{2,\frk{q}})$. This suggests the possibility of incorporating inert primes as well in our collection of classes.
\end{rmk}

\subsection{Selmer groups and local conditions at \texorpdfstring{$p$}{p}}
Now assume that $p$ is inert in $K$ and we identify the completion of $K$ at $p$ with $\Q_{p^2}$ (the unique degree $2$ unramified extension of $\Q_p$ in our fixed algebraic closure $\bar{\Q}_p$). We also assume that $f\in S_k(\Gamma_0(N_f))$ (with $k\geq 2$ even integer) is ordinary at $p$. Then the restriction of $V(f)$ to a decomposition group at $p$ (which we identify with $G_{\Q_p}$ via our fixed embedding $\iota_p$) admits a filtration
\begin{equation*}
0 \longrightarrow V(f)^+ \longrightarrow V(f) \longrightarrow V(f)^- \longrightarrow 0
\end{equation*}
where $V(f)^\pm$ is one-dimensional and $V(f)^-$ is unramified with $\Frob_p$ acting as multiplication by $\alpha_f$, the unit root of the Hecke polynomial of $f$ at $p$.

Let $\chi$ denote an anticyclotomic Hecke character of $K$ of $\infty$-type $(-j,j)$ with $j\in\Z_{\geq 0}$ and with conductor coprime to $p$.

\begin{nota}
We define $V_{f,\chi}:=V(f)(1-k/2)\otimes \hat{\chi}^{-1}$ (and similarly for $T_{f,\chi}$), where, as before, $\hat{\chi}$ is the $p$-adic Galois character attached to the $p$-adic avatar of $\chi$.
\end{nota}

We define $G_{\Q_{p^2}}$-stable subspaces
\begin{equation}\label{eq:bal-Sh}
\cl{F}_p^{+,\ord}(V_{f,\chi}):=V(f)^+(1-k/2)\otimes\hat{\chi}^{-1},\quad
\cl{F}_p^{+,\rel}(V_{f,\chi}):=V_{f,\chi},
\end{equation}
and, for $?\in\{\ord,\rel\}$, we let $\cl{F}_p^{-,?}(V_{f,\chi}):=V_{f,\chi}/\cl{F}_p^{+,?}(V_{f,\chi})$.

Let $L/K$ be a finite extension, and for every finite prime $v$ of $L$ and for $?\in\{\rel,\ord\}$, put
\[
H^1_?(L_v,V_{f,\chi}):=
\begin{cases}
{\rm ker}\big(H^1(L_v,V_{f,\chi})\rightarrow H^1(L_v,\cl{F}_p^{-,?}(V_{f,\chi}))\big)&\textrm{if $v\mid p$},\\
{\rm ker}\big(H^1(L_v,V_{f,\chi})\rightarrow H^1(L_v^{\rm nr},V_{f,\chi})\big)&\textrm{if $v\nmid p$,}
\end{cases}
\]
where $L_v^{\rm nr}$ is the maximal unramified extension of $L_v$. 
We then let $H^1_?(L_v,T_{f,\chi})$ be the inverse image of $H^1_?(L_v,V_{f,\chi})$ under the natural map $H^1(L_v,T_{f,\chi})\rightarrow H^1(L_v,V_{f,\chi})$, and let ${\rm Sel}_?(L,V_{f,\chi})\subset H^1(L,V_{f,\chi})$ (resp. ${\rm Sel}_?(L,T_{f,\chi})\subset H^1(L,T_{f,\chi})$) be the \emph{ordinary} (for $?=\ord$) or \emph{relaxed} (for $?=\rel$) Greenberg Selmer groups cut out by these local conditions. Explicitly,
\[
\Sel_?(L,V_{f,\chi}):={\rm ker}\Bigg(H^1(L,V_{f,\chi})\to\prod_v\frac{H^1(L_v,V_{f,\chi})}{H^1_?(L_v,V_{f,\chi})}\Bigg),
\]
where $v$ runs over the finite places of $L$.

Following \cite{BK1990}, we recall the definition of the Bloch-Kato Selmer group $\Sel_{\BK}(L,V_{f,\chi})$
\[
\Sel_{\BK}(L,V_{f,\chi}):={\rm ker}\Bigg(H^1(L,V_{f,\chi})\to\prod_v\frac{H^1(L_v,V_{f,\chi})}{H^1_f(L_v,V_{f,\chi})}\Bigg),
\]
where again $v$ runs over the finite places of $L$ and
\[
H^1_f(L_v,V_{f,\chi}):=
\begin{cases}
{\rm ker}\big(H^1(L_v,V_{f,\chi})\rightarrow H^1(L_v,V_{f,\chi}\otimes\bf{B}_{\cris})\big)&\textrm{if $v\mid p$},\\
{\rm ker}\big(H^1(L_v,V_{f,\chi})\rightarrow H^1(L_v^{\rm nr},V_{f,\chi})\big)&\textrm{if $v\nmid p$,}
\end{cases}
\]
with $\bf{B}_{\cris}$ denoting Fontaine's crystalline period ring. As usual, we define local conditions $H^1_f(L_v,T_{f,\chi})$ by propagation and obtain the integral Bloch-Kato Selmer group $\Sel_{\BK}(L,T_{f,\chi})$.

It is obviously true that $\Sel_{\BK}(L,V_{f,\chi})\subseteq\Sel_{\rel}(L,V_{f,\chi})$ for any finite extension $L$ of $K$.

\begin{lemma}
\label{lemma: descriptionSelmer}
If $\chi$ has $\infty$-type $(-j,j)$ with $0\leq j<k/2$, then $\Sel_{\ord}(K,V_{f,\chi})=\Sel_{\BK}(K,V_{f,\chi})$ and $\Sel_{\ord}(K[n],V_{f,\chi})=\Sel_{\BK}(K[n],V_{f,\chi})$ for all $n\geq 1$ coprime to $p$.
\begin{proof}
Write $V=V(f)(1-k/2)\otimes\Ind_K^\Q\hat{\chi}^{-1}$ till the end of the proof. Shapiro's lemma affords an identification $H^1(K,V_{f,\chi})=H^1(\Q,V)$, under which $H^1_{\ord}(\Q_{p^2},V_{f,\chi})$ is identified with
\[
{\rm Im}\Big(H^1\big(\Q_p,V(f)^+(1-k/2)\otimes\Ind_K^\Q\hat{\chi}^{-1}\big)\to H^1(\Q_p,V)\Big).
\]
It is then an exercise to see that $V^+:=V(f)^+(1-k/2)\otimes\Ind_K^\Q\hat{\chi}^{-1}$ is a $G_{\Q_p}$-stable subspace whose Hodge-Tate weights are positive (they are $k/2-j$ and $k/2+j$), such that the quotient $V/V^+$ has non-positive Hodge-Tate weights (they are $1-k/2-j$ and $1-k/2+j$). The lemma follows directly from \cite[Lemma 2]{Fla1990} and the fact that
$H^1_e(\Q_p,V)=H^1_f(\Q_p,V)=H^1_g(\Q_p,V)$ (this last statement being a consequence of the easily checked fact that ${\bf D}_{\cris}(\Q_p,V)^{\varphi=1}=\{0\}$, cf. \cite[Corollary 3.8.4]{BK1990}). The equality $\Sel_{\ord}(K[n],V_{f,\chi})=\Sel_{\BK}(K[n],V_{f,\chi})$ for $n\geq 1$ coprime to $p$ follows in the same way, observing that $p$ splits completely in the ring class field of $K$ of conductor $n$, and thus in $K[n]$ as well. Repeating the above proof after localization at every finite place of $v$ of $K[n]$ dividing $p$, one obtains the result.
\end{proof}
\end{lemma}

Now we go back to the setting of Section \ref{subsec:tameeulersystem}, from which we borrow the notation.

\begin{proposition}\label{prop:in-Sel}
For every $n\in\mathcal{S}$ and for $i=1,2$ the class $\kappa_{n,f,\eta_i}$ lies in the Bloch-Kato Selmer group $\Sel_{\BK}(K[n],V_{f,\eta_i})$. In particular, 
the class $\kappa_{n,f,\eta_2}$ lies in $\Sel_{\ord}(K[n],V_{f,\eta_2})$.
\begin{proof}
Fix $n\in\mathcal{S}$ and $v$ a finite prime of $K[n]$. If $v\nmid p$, then 
it follows from the Weil conjectures that $V_{f,\eta_i}$ is pure of weight $-1$, and hence
\begin{equation}\label{eq:ur=0}
H^1_f(K[n]_v,V_{f,\eta_i})={\rm ker}\bigl(H^1(K[n]_v,V_{f,\eta_i})\rightarrow H^1(K[n]_v^{\rm nr},V_{f,\eta_i})\bigr)=0.
\end{equation}
By \cite[Corollary 1.3.3(i)]{Rub2000} and local Tate duality (using the fact that the $G_K$-representations $V_{f,\eta_i}$ are conjugate self-dual), it follows that
\[
H^0(K[n]_v,V_{f,\eta_i})=H^2(K[n]_{\overline{v}},V_{f,\eta_i})=0.
\]
Repeating the argument switching the roles of $v$ and $\overline{v}$, from (\ref{eq:ur=0}) and \cite[Corollary 1.3.3(ii)]{Rub2000} we obtain
\[
H^1(K[n]_v,V_{f,\eta_i})=H^1_f(K[n]_v,V_{f,\eta_i})=0. 
\]
It follows that the inclusion ${\rm res}_v(\kappa_{n,f,\eta_i})\in H^1_f(K[n]_v,V_{f,\eta_i})$ is automatic for $i\in\{1,2\}$.

Now suppose $v\mid p$. As noted in \cite[Proposition 3.2]{BSVast1}, the classes $\kappa_{Nn,\mathbf{r}}$ are geometric at $p$. This applies to the classes $\kappa_{N,\Delta_n,\mathbf{r}}$ as well. Hence, the classes ${\rm res}_v(\kappa_{n,f,\eta_i})\in H^1(K[n]_v,V_{f,\eta_i})$ land in $H^1_g(K[n]_v,V_{f,\eta_i})$. As remarked in the proof of the above Lemma \ref{lemma: descriptionSelmer},
$H^1_g(K[n]_v,V_{f,\eta_i})$ agrees with the Bloch-Kato subspace $H^1_f(K[n]_v,V_{f,\eta_i})$ in our case, so that indeed the class $\kappa_{n,f,\eta_i}$ lies in the Bloch-Kato Selmer group $\Sel_{\BK}(K[n],V_{f,\eta_i})$ for $i\in\{1,2\}$. 

We have that $\eta_2$ has $\infty$-type $(-(l-m)/2,(l-m)/2)$, so that (thanks to the balanced condition on $(k,l,m)$) the above Lemma \ref{lemma: descriptionSelmer} applies and we have $\Sel_{\BK}(K[n],V_{f,\eta_2})=\Sel_{\ord}(K[n],V_{f,\eta_2})$.
\end{proof}
\end{proposition}

Till the end of the section, let us denote by $L$ a finite extension of $K$ containing the Fourier coefficients of $f$ and the values of the Hecke character $\chi$. We also let $\mathfrak{P}$ be the prime of $L$ above $p$ induced by the fixed choice of embedding $\iota_p:\bar{\Q}\hookrightarrow\bar{\Q}_p$ and we think of our field of coefficients $E$ as $E=L_{\mathfrak{P}}$ (the completion of $L$ at $\mathfrak{P}$). 

\begin{defi}
\label{def: big image}
We say that $f$ has big image at $\mathfrak{P}$ if the image of $G_{\Q}$ in $\Aut_{\cl{O}}T(f)$ contains a conjugate of $\SL_2(\Z_p)$. We also say that $f$ has big image at $p$ if we do not want to be explicit about the choice of $L$ and $\mathfrak{P}$.    
\end{defi}

\begin{lemma}
Assume that the modular form $f$ is not of CM type. Then the $G_K$-representation $V:=V_{f,
\chi}$ (with $\cl{O}$-lattice $T=T_{f,\chi}$) satisfies the following properties:
\begin{enumerate}[(i)]
    \item $V$ is absolutely irreducible as a representation of $G_K$;
    \item there exists $\sigma\in\Gal(\bar{K}/K^{\mathrm{ ab}})$ such that $\dim_E\big(V/(\sigma-1)V\big)=1$;
    \item there exists $\gamma\in\Gal(\bar{K}/K^{\mathrm{ab}})$ such that $V^{\gamma=1}=0$.
\end{enumerate}
If, moreover, $f$ has big image at $\mathfrak{P}$, the $\cl{O}$-module $H^1(K[n],T)$ is $\cl{O}$-free for all $n\geq 1$.
\begin{proof}
This is a direct consequence of \cite[Propositions 7.1.4 and 7.1.6]{LLZ2015}
\end{proof}
\end{lemma}
Note that it is well-known that if $f$ is not of CM type, then it has big image at almost all finite primes of $L$.

\medskip
Recall that, since we assume that $p\nmid h_K$ ($h_K$ being the class number of $K$), we have $K[1]=K$. 

\begin{nota}
\label{nota: bottom}
With the notation of Section \ref{subsec:tameeulersystem}, we write $\kappa_{f,\eta_i}:=\kappa_{1,f,\eta_i}$.
\end{nota}

\begin{rmk}
Note that (by Theorem \ref{principal:tame}) it follows that $\kappa_{f,\eta_i}=\tilde{\kappa}_{1,f,\eta_i}$.    
\end{rmk}

\medskip
Forthcoming work of Jetchev-Nekov\'{a}\v{r}-Skinner (cf. \cite[Section 8.1]{ACR2023} for an overview) develops the theory of split anticyclotomic Euler systems (or \emph{JNS Euler systems}). This should lead to the following result.

\begin{corollary}
\label{cor: rank1}
In the setting of Theorem \ref{principal:tame}, assume moreover that $p$ is inert in $K$ and that $f$ is not of CM type and has big image at $p$. Then:
\begin{itemize}
    \item [(i)] If $\kappa_{f,\eta_1}\neq 0$, then $\dim_E \Sel_{\BK}(K,V_{f,\eta_1})=1$;
    \item [(ii)] If $\kappa_{f,\eta_2}\neq 0$, then $\dim_E\Sel_{\BK}(K,V_{f,\eta_2})=1$.
\end{itemize}
\end{corollary}



\section{The relevant \texorpdfstring{$p$}{p}-adic \texorpdfstring{$L$}{L}-functions}
\label{sec: padicL}
In this section we introduce the $p$-adic $L$-functions that will play a role in the sequel of the paper. Most of the results come from our previous work \cite{Mar2024a}.

\subsection{Hida families}\label{subsubsec:Hida}
Let $\Lambda =  \mathbb Z_p[\![1+p\mathbb Z_p]\!]$ and let $\cW=\rm{Spf} (\Lambda)^{\rm rig}$
be the weight space. For any finite extension $E$ of $\bb{Q}_p$ with ring of integers $\cl{O}=\cl{O}_E$, we write $\Lambda_E:=\cl{O}_E[\![1+p\mathbb Z_p]\!]$. 

\medskip
We have $\cW(E)=\Hom_{\text{cont}}(1+p\bb{Z}_p,E^\times)$. Points of the form $\nu_{k,\epsilon}(n)= \epsilon(n)n^{k-2}$, where $k\in\Z_{\geq 1}$ and $\epsilon$ is a finite order character, will be called \emph{arithmetic}. We refer to $k$ as the \emph{weight} of $\nu_{k,\epsilon}$. Arithmetic points of the form $\nu_{k,\epsilon}$ with $k\geq 2$ will be called \emph{classical}.

More generally, let $\cl{R}$ be a noetherian local ring which is a flat $\Lambda$-algebra and let $\cW_{\cl{R}}=\rm{Spf}(\cl{R})^{\rm rig}$. A point $x\in\cW_{\cl{R}}(E)$ will be called \emph{arithmetic} if it lies above an arithmetic point $\nu_{k,\epsilon}$ and \emph{classical} if it lies above a classical point. Again, we refer to $k$ as the \emph{weight} of $x$.

Let $N$ be a positive integer coprime to $p$. A \emph{Hida family} of tame level $N$ and character $\chi:(\bb{Z}/Np\bb{Z})^\times\rightarrow E^\times$ is a formal $q$-expansion
\[
\Hf = \sum_{n\geq 1} a_n(\hf) q^n \in \Lambda_{\hf}[\![q]\!],
\]
where $\Lambda_\Hf$ is a finite flat $\Lambda$-algebra, such that, for any classical point $x$ lying over some $\nu_{k,\epsilon}$, the corresponding specialization is a $p$-ordinary eigenform $\Hf_x\in S_k(Np^s,\chi\epsilon\omega^{2-k})$. As above, we have denoted by $k$ the weight of $x$ and we can take $s=\max\lbrace 1, \ord_p(\cond(\epsilon))\rbrace$. We say that a Hida family $\hf$ is \emph{primitive} if the specializations $\Hf_x$ at arithmetic points $x$ are $p$-stabilized newforms. We say that it is \emph{normalized} if $a_1(\hf)=1$.

Let $\Hf$ be a normalized primitive Hida family of tame level $N_\Hf$. For each arithmetic point $x\in \cW_{\Lambda_\hf}(\overline{\bb{Q}}_p)$, let $\Hf_x$ denote the specialization of $\Hf$ at $x$ and let $f_x$ be the corresponding newform. There exists a locally-free rank-two $\Lambda_\Hf$-module $\mathbb{V}_\Hf$ equipped with a continuous action of $G_\bb{Q}$ such that, for any arithmetic point $x\in \cW_{\Lambda_\Hf}(E)$, the corresponding specialization $\bb{V}_\Hf\otimes_{\Lambda_\Hf,x}E$ recovers the $G_\bb{Q}$-representation $V(f_x)$ attached to $f_x$. In particular, the representation $\bb{V}_\Hf$ is unramified at any prime $q\nmid p\cdot N_\Hf$ and $\Tr(\Fr_q)=a_q(\Hf)$. We refer to $\bb{V}_\Hf$ as the \emph{big Galois representation attached to $\Hf$}. If for some (equivalently all) arithmetic point $x_0\in \cW_{\Lambda_\Hf}(\overline{\bb{Q}}_p)$ the $G_\bb{Q}$-representation $T(f_{x_0})$ attached to $f_{x_0}$ is residually irreducible, then $\bb{V}_\Hf$ is a free $\Lambda_\Hf$-module.

\subsection{Families of theta series of infinite \texorpdfstring{$p$}{p}-slope}
\label{subsection: theta families}
In \cite[Section 2]{Mar2024a}, we defined \emph{generalized} $\Lambda$-adic eigenforms as suitable formal $q$-expansions in $\cl{R}[\![q]\!]$ where $\cl{R}$ is a noetherian local ring which is an integral domain and a flat $\Lambda$-algebra (but it is not necessarily finite as $\Lambda$-algebra). Such formal $q$-expansions satisfy the property that there exists a \emph{large} subset $\Omega$ of integral classical weights in $\cl{W}_{\cl{R}}$ such that for every $x\in\Omega$, the specialization at $x$ is the $q$-expansion of a classical eigenform (we drop the ordinarity condition).

\medskip
The motivation to introduce such a notion was afforded by the construction of generalized $\Lambda$-adic eigenforms interpolating theta series attached to Hecke characters of a quadratic imaginary field $K$ where $p$ is inert.
We identify the completion of $K$ at $p$ with $\Q_{p^2}$ (with ring of integers $\Z_{p^2}$). Recall the decomposition
\[
\Z_{p^2}^\times=\mu_{p^2-1}\times (1+p\Z_{p^2})
\]
induced by the Teichm\"{u}ller lift. As before, we keep assuming that $p\nmid h_K$ (where $h_K$ is the class number of $K$). 

\medskip
We let $K_\infty$ denote the (unique) $\Z_p^2$-extension of $K$ and we let $\Gamma_\infty:=\Gal(K_\infty/K)$.

\medskip
We refer to \cite[Section 2.1]{Mar2024a} (note that the conventions for class field theory and infinity type are not the same as in this article) for the construction of the (unique) Hecke character $\psi_0$ (denoted $\langle\lambda\rangle$ in \cite{Mar2024a}) of $K$ of $\infty$-type $(-1,0)$ and conductor $(p)$ whose $p$-adic avatar factors through $\Gamma_\infty$ and yields an identification $\Gamma_\infty\cong 1+p\Z_{p^2}$. Existence and uniqueness of such a character follow from our assumption on the class number of $K$

\medskip
Given an algebraic Hecke character $\psi$ of $K$ of $\infty$-type $(1-\nu_0,0)$ (for some $\nu\in\Z_{\geq 1}$) and conductor $\frk{f}$ (coprime to $p$), we can write $\psi=\xi\cdot\psi_0^{\nu_0-1}$ as a Hecke character modulo $\frk{f}p$, with $\xi$ a ray class character of conductor dividing $\frk{f}p$. It is then possible to obtain $p$-adic families of theta series of infinite $p$-slope passing through $\theta_{\psi}$ as follows.

The coefficient ring for such a family will be the $\Lambda$-algebra $\Lambda_\infty:=\cl{O}_E[\![\Gamma_\infty]\!]$, with $\Lambda$-algebra structure induced by the natural inclusion $1+p\Z_p\subset 1+p\Z_{p^2}\cong\Gamma_\infty$.

\begin{defi}
\label{defGHida}
We define
\[
\pmb{\theta}_\xi:=\sum_{(\mathfrak{a},\frk{f}p=1)}\xi\psi_0(\mathfrak{a})[\mathfrak{a}]\cdot q^{N_{K/\Q}(\mathfrak{a})}\in\Lambda_\infty[\![q]\!].
\]
Here $[\frk{a}]$ denotes the projection to $\Gamma_\infty$ of the element of the ray class group of conductor $\frk{f}p^\infty$ which corresponds to the class of the ideal $\frk{a}$ via class field theory. 
\end{defi}

Specializing such a $q$-expansion at points of $\cl{W}_{\Lambda_\infty}$ which correspond to the characters $\Gamma_\infty\to\C_p^\times$ of the form $\gamma\mapsto\hat{\psi}_0(\gamma)^{\nu-2}$ for integers $\nu\geq 1$, one obtains the $q$-expansion of the theta series associated with the character $\xi\psi_0^{\nu-1}$, which is an eigenform of level $N_{K/\Q}(\frk{f})\cdot D_K\cdot p^2$ and weight $\nu$. In particular, for $\nu=\nu_0$, we obtain the $p$-depletion $\theta_\psi^{[p]}$ of $\theta_\psi$. This essentially shows that $\boldsymbol{\theta}_\xi$ is a generalized $\Lambda$-adic eigenform (cf. \cite[Lemma 4.21]{Mar2024a}).

\subsection{Factorization of triple product \texorpdfstring{$p$}{p}-adic \texorpdfstring{$L$}{L}-functions}
\label{factorizationsectionintro}
Assume that $p\geq 5$ and let $\Hf$ be a Hida family of tame level $N_\Hf$ with trivial tame character and coefficients in $\Lambda_\Hf$. We can and will assume that $\Lambda_\Hf$ is a finite flat $\Lambda_E$-algebra. We also impose the following assumption on $\Hf$.

\begin{ass}[CR]
\label{CRass}
The residual Galois representation $\bar{\bb{V}}_\Hf$ of the big Galois representation $\bb{V}_\Hf$ attached to $\Hf$ is absolutely irreducible and $p$-distinguished.
\end{ass}

Fix $K/\Q$ a quadratic imaginary field of odd discriminant $-D_K$ and two ray class characters $\xi_1$ and $\xi_2$ of $K$, that we can view as valued in $E$.

\begin{ass}
\label{CMass}
The following assumptions are in force.
\begin{enumerate}[(i)]
    \item $p$ is inert in $K$.
    \item $N_\Hf$ is squarefree, coprime to $D_K$ and with an odd number of prime divisors which are inert in $K$ (we are in the so-called \emph{definite} case).
    \item $\xi_i$ has conductor dividing $(cp)$, with $c\in\Z_{\geq 1}$, $(c, p\cdot D_K\!\cdot\! N_\Hf)=1$, $c$ not divisible by primes inert in $K$.
    \item $\xi_1$ and $\xi_2$ are not induced by Dirichlet characters and the central characters of $\xi_1$ and $\xi_2$ are inverse to each other, so that $\alpha:=\xi_1\xi_2$ and $\beta:=\xi_1\xi_2^\mathbf{c}$ are ring class characters of $K$ (here $\langle \mathbf{c}\rangle =\Gal(K/\Q)$).
\end{enumerate}    
\end{ass}

\medskip

We let $\Hg:=\pmb{\theta}_{\xi_1}$ and $\Hh:=\pmb{\theta}_{\xi_2}$. In \cite[Section 5]{Mar2024a} we defined an improved generalized triple product $p$-adic $L$-function, denoted $\mathcal{L}_p^f(\Hf,\Hg,\Hh)$, whose square interpolates the (algebraic part of the) central values for the triple product $L$-function attached to triples of specializations of $(\Hf,\Hg,\Hh)$ in the $\Hf$-dominant range (i.e., the weight of the specialization of $\Hf$ is at least equal to the sum of the weights of the specializations of $\Hg$ and $\Hh$).

The $p$-adic $L$-function $\mathcal{L}_p^f(\Hf,\Hg,\Hh)$ is an element of the ring
\[
\cl{R}_{\Hf\Hg\Hh}:=\Lambda_\Hf\hat{\otimes}_{\cl{O}_E}\Lambda_\infty\hat{\otimes}_{\cl{O}_E}\Lambda_\infty
\]
and depends on a choice of \emph{test vectors} for the families $\Hg$ and $\Hh$ and of a generator of the so-called congruence ideal of $\Hf$ (Assumption \ref{CRass} ensures indeed that the congruence ideal of $\Hf$ is principal).

\begin{rmk}
\label{rmk: notseriousass}
Typically in point (ii) of Assumption \ref{CMass} one only asks that $N_{\Hf}$ admits a factorization $N_{\Hf}=N_{\Hf}^+ N_{\Hf}^-$ with $N_{\Hf}^+$ (resp. $N_{\Hf}^-$) only divisible by primes split (resp. inert) in $K$ and with $N_{\Hf}^-$ squarefree with an odd number of prime divisors. We have inherited the stronger condition that $N_{\Hf}$ is squarefree from \cite{Mar2024a}, where this is used merely to simplify the shape of the \emph{test vectors} that appear in the construction of the improved triple product $p$-adic $L$-function. A similar discussion applies to the point (iii) of Assumption \ref{CMass}, concerning the fact that $\xi_i$ are supposed to have conductor divisible only by primes split in $K$. With some extra work and notation we could certainly allow $N_\Hf^+$ non squarefree and the conductor of $\xi_i$ divisible by primes inert in $K$.
\end{rmk}

Let $H_n$ denote the ring class field of $K$ of conductor $cp^n$ for every $n\in\Z_{\geq 0}$ and let $H_\infty$ be the union of all the $H_n$'s. Let $\mathscr{G}_\infty:=\Gal(H_\infty/K)$. We can identify the maximal $\Z_p$-free quotient $\Gamma^-$ of $\mathscr{G}_\infty$ with the Galois group of the anticyclotomic $\Z_p$-extension of $K$ and there is an exact sequence $0\to\Delta_c\to\mathscr{G}_\infty\to\Gamma^-\to 0$ of abelian groups with $\Delta_c$ a finite group and $\Gamma^-\cong\Z_p$. We fix a non-canonical isomorphism $\mathscr{G}_\infty\cong\Delta_c\times\Gamma^-$ once and for all.

Then $\alpha$ (resp. $\beta$) factors through $\mathscr{G}_\infty$ and we write it as $(\alpha_t,\alpha^-)$ (resp. $(\beta_t,\beta^-)$) according to the fixed isomorphism $\mathscr{G}_\infty\cong\Delta_c\times\Gamma^-$.

For $k\in\Z_{\geq 2}\cap 2\Z$, let $\mathfrak{X}^{\mathrm{crit}}_{p,k}$ denote the set of continuous characters $\hat{\nu}:\Gamma^-\to\C_p^{\times}$ such that the associated algebraic Hecke character $\nu:\bb{A}_K^\times/K^\times\to\C^\times$ has infinity type $(j,-j)$ with $|j|<k/2$. 

\medskip
The main result of \cite[Section 5]{Mar2024a} is the following factorization theorem for the \emph{anticyclotomic projection} $\mathcal{L}_{p,ac}^f(\Hf,\Hg,\Hh)$ of $\mathcal{L}_p^f(\Hf,\Hg,\Hh)$, which is obtained mapping $\mathcal{L}_p^f(\Hf,\Hg,\Hh)$ to an element of $\Lambda_{\Hf}\hat{\otimes}_{\cl{O}_E}\cl{O}_E[\![\Gamma^-]\!]\hat{\otimes}_{\cl{O}_E}\cl{O}_E[\![\Gamma^-]\!]$ as follows.

\medskip
We consider the automorphism $\sigma$ of $\Lambda_\infty\hat{\otimes}_{\cl{O}_E}\Lambda_\infty$ given by the assignment
\[
[\gamma]\otimes[\delta]\mapsto[\gamma^{1/2}\delta^{1/2}]\otimes[\gamma^{1/2}\delta^{-1/2}]
\]
on group-like elements (note it is important that $p\neq 2$ for this to be a well-defined automorphism). The natural projection $\Gamma_\infty\twoheadrightarrow\Gamma^-$ can be described as $\gamma\mapsto\gamma^{1/2}(\gamma^\mathbf{c})^{-1/2}$. Accordingly, we get a morphism $\tau\colon\Lambda_\infty\twoheadrightarrow\cl{O}_E[\![\Gamma^-]\!]$. Then $\mathcal{L}_{p,ac}^f(\Hf,\Hg,\Hh)$ is obtained applying $(1_{\Lambda_\Hf},(\tau,\tau)\circ\sigma)$ to $\mathcal{L}_p^f(\Hf,\Hg,\Hh)$.

\begin{theorem}[cf. Theorem 5.25 of \cite{Mar2024a}]
\label{factorizationthmintro}
In the above setting, it holds:
\[
\mathcal{L}^f_{p,ac}(\Hf,\Hg,\Hh)=\pm\mathscr{A}_{\Hf\Hg\Hh}\cdot\Big(\alpha^-\big(\Theta_\infty^{\mathrm{Heeg}}(\Hf,\alpha_t)\big)~\hat{\otimes}~\beta^-\big(\Theta_\infty^{\mathrm{Heeg}}(\Hf,\beta_t)\big)\Big).
\]
\end{theorem}

\begin{rmk}
\label{factorizationrmk}
This equality takes place in the ring
\[
\mathscr{R}^-:=(\cl{R}_{\Gamma^-}\hat{\otimes}_{\Lambda_\Hf} \cl{R}_{\Gamma^-})[1/p],\quad\text{where}\quad \cl{R}_{\Gamma^-}:=\Lambda_\Hf\hat{\otimes}_{\cl{O}_E}\cl{O}_E[\![\Gamma^-]\!]
\]
and the notation is as follows.
\begin{enumerate}[(i)]
    \item $\Theta_\infty^{\mathrm{Heeg}}(\Hf,\alpha_t)\in \cl{R}_{\Gamma^-}$ (resp. $\Theta_\infty^{\mathrm{Heeg}}(\Hf,\beta_t)\in \cl{R}_{\Gamma^-}$) is (a slight generalization of) the so-called \emph{big theta element} constructed by Castella-Longo in \cite{CL2016}, building up on works by Bertolini-Darmon (cf. \cite{BD1996},\cite{BD1998},\cite{BD2007}) and Chida-Hsieh (cf. \cite{CH2018}). These $p$-adic $L$-functions interpolate the (square root of the algebraic part of the) special values $L(\Hf_{\!k}/K,\alpha_t\nu,k/2)$ (resp. $L(\Hf_{\!k}/K,\beta_t\nu,k/2)$) for $k\in\Z_{\geq 2}$ even and $\hat{\nu}\in\mathfrak{X}^{\mathrm{crit}}_{p,k}$.
    \item $\alpha^-(w)$ (resp. $\beta^-(w)$) for $w\in \cl{R}_{\Gamma^-}$ denotes the image of the element $w$ via the $\cl{O}_E$-linear automorphism of $\cl{R}_{\Gamma^-}$ uniquely determined by the identity on $\Lambda_E$ and the assignment $[\gamma]\mapsto\alpha^-(\gamma)[\gamma]$ (resp. $[\gamma]\mapsto\beta^-(\gamma)[\gamma]$) on group-like elements on $\cl{O}_E[\![\Gamma^-]\!]$.
    \item The element $\mathscr{A}_{\Hf\Hg\Hh}\in\mathscr{R}^-$ is described more precisely in \cite[Proposition 5.23]{Mar2024a}. Here we are in a slightly different setting, but it is easy to see that, if one fixes a specific $k\in\Z_{\geq 2}\cap 2\Z$, one can run all the constructions in such a way that $\mathscr{A}_{\Hf\Hg\Hh}(k,-,-)\neq 0$ (independently of the specializations in the last two variables). Actually one can make even sure that $\mathscr{A}_{\Hf\Hg\Hh}(k',-,-)\neq 0$ for all $k'$ is a small enough neighbourhood of $k$ in the weight space.
\end{enumerate}    
\end{rmk}

\section{The explicit reciprocity law}
\label{section: ERL}
In this section, we explain how to reinterpret the explicit reciprocity law for diagonal classes obtained in \cite[Theorem A]{BSV2020a} in our setting. 

\subsection{Explicit reciprocity law for diagonal classes}
Let as usual $E$ denote a finite and large enough extension of $\Q_p$ containing with ring of integers $\cl{O}$. For the moment, let $(f,g,h)$ denote any triple of normalized cuspforms
\[
f\in S_k(N,\chi_f,E),\quad g\in S_l(N,\chi_g,E),\quad h\in S_m(N,\chi_h,E)
\]
of some common level $N\geq 5$ coprime to $p$. We assume that $f,g,h$ are eigenforms for the \emph{good} Hecke operators $T_\ell$ for primes $\ell\nmid N$ (we do not assume that they are newforms, though). We also impose the self-duality assumption
\[
\chi_f\cdot\chi_g\cdot\chi_h=1.
\]
and the balanced condition on the triple of weights $(k,l,m)\in(\Z_{\geq 2})^3$, i.e., $(k,l,m)$ are the sizes of the edges of a triangle. We write $\mathbf{r}=(k-2,l-2,m-2)$ and $r=(k+l+m-6)/2$ in what follows and we assume $p+2\geq\max\{k,l,m\}$.

\medskip
As recalled in Section \ref{subsec: improved diagonal classes}, one can construct a class
\[
\kappa_{N,\mathbf{r}}:=\kappa_{N(1),\mathbf{r}}\in H^1\big(\Q,H^3_{\et}(Y_1(N)^3_{\bar{\Q}},\Ls_{[\mathbf{r}]}(2-r))_{\Q_p}\big).
\]

Applying the K\"{u}nneth decomposition and the $(f,g,h)$-isotypic projection, one obtains a class $\kappa_N(f,g,h)\in H^1(\Q,V_N(f,g,h))$, where (following the notation of Section \ref{subsec: Galoisreps}) we set
\[
T_N(f,g,h):=T_N(f)\otimes_{\cl{O}}\otimes T_N(g)\otimes_{\cl{O}}T_N(h)(-1-r).
\]
and $V_N(f,g,h)=T_N(f,g,h)\otimes_{\cl{O}}E$.

As explained in \cite[Proposition 3.2]{BSVast1}, the restriction of class $\kappa_N(f,g,h)$ to a decomposition group at $p$ lies in the geometric Bloch-Kato subspace $H^1_g(\Q_p,V_N(f,g,h))$. Moreover, in this setting one can rather easily check (cf. \cite[Lemma 3.5]{BSVast1}) that one has
\[
H^1_e(\Q_p,V_N(f,g,h))=H^1_f(\Q_p,V_N(f,g,h))=H^1_g(\Q_p,V_N(f,g,h))
\]
for the three Bloch-Kato local conditions at $p$. 

\medskip
In particular, one can apply the Bloch-Kato logarithm to the class $\kappa_N(f,g,h)$
\[
\log_{\BK}:H^1_e(\Q_p,V_N(f,g,h))\xrightarrow{\cong} \frac{{\bf D}_{\dR}(V_N(f,g,h))}{\Fil^0 {\bf D}_{\dR}(V_N(f,g,h))}\cong \Fil^0{\bf D}_{\dR}(V^*_N(f,g,h))^\vee
\]
and view it as a linear functional
\[
\log_{\BK}(\kappa_N(f,g,h)):\Fil^0{\bf D}_{\dR}(V^*_N(f,g,h))\to E.
\]
Here $V^*_N(f,g,h):=V^*_N(f)\otimes_E V^*_N(g)\otimes_E V^*_N(h)(r+2)$ (again with the notation of Section \ref{subsec: Galoisreps}) and $(-)^\vee$ means taking the $E$-dual vector space.

\medskip
As explained in \cite[Section 2.5]{BSVast1}, the comparison isomorphisms of Faltings and Tsuji yield a canonical isomorphism
\[
\Fil^1{\bf D}_{\dR}(V^*_N(g))\cong S_l(\Gamma_1(N),E)[\xi]
\]
and similarly for $h$. Hence we can consider the element $\omega_g\in \Fil^1{\bf D}_{\dR}(V^*_N(g))$ corresponding to $g$ under the above identification. The same procedure gives a differential $\omega_h$ attached to $h$.

\begin{ass}
From now on in this section we assume that the eigenform $f$ is $p$-ordinary.
\end{ass}

Under the above assumption, one can obtain a decomposition
\begin{equation}
\label{dRdecompf}
{\bf D}_{\dR}(V^*_N(f))=\Fil^1{\bf D}_{\dR}(V^*_N(f))\oplus {\bf D}_{\dR}(V^*_N(f))^{\varphi=\alpha_f}   
\end{equation}
where $\alpha_f\in\cl{O}^\times$ is the root of the Hecke polynomial at $p$ for $f$ which is a $p$-adic unit and $\varphi$ is the crystalline Frobenius acting on ${\bf D}_{\dR}(V^*_N(f))={\bf D}_{\cris}(V^*_N(f))$.

\medskip
The $p$-adic comparison isomorphisms between \'{e}tale and de Rham cohomology also yield an identification
\begin{equation}
\label{vdrsufil1}
\frac{{\bf D}_{\dR}(V^*_N(f))}{\Fil^1 {\bf D}_{\dR}(V^*_N(f))}\cong S_k(\Gamma_1(N),E)[f^w]^\vee, 
\end{equation}
where, for any $\xi\in S_\nu(\Gamma_1(N),E)$, $\xi^w$ denotes the image of $\xi$ under the Atkin-Lehner operator $w_N$.

\medskip
One can thus obtain an element $\eta_f^{\alpha}\in {\bf D}_{\dR}(V^*_N(f))^{\varphi=\alpha_f}$ as the unique lift to ${\bf D}_{\dR}(V^*_N(f))^{\varphi=\alpha_f}$ of the element of ${\bf D}_{\dR}(V^*_N(f))/\Fil^1 {\bf D}_{\dR}(V^*_N(f))$ which corresponds, under the identification \eqref{vdrsufil1}, to the linear functional 
\[
S_k(\Gamma_1(N),E)[f^w]\to E\qquad \xi\mapsto \frac{\langle \xi,f^w\rangle_{Pet}}{\langle f^w,f^w\rangle_{Pet}}.
\]
Here $\langle -,-\rangle_{Pet}$ denotes Petersson inner product (normalized as in \cite[Definition 3.4]{Mar2024a}).

\medskip
Let $\{t_n\}_{n\in\Z}$ denote a compatible choice of generators for ${\bf D}_{\dR}(\Q_p(n))$ for varying $n\in\Z$, corresponding to a choice of a uniformizer $t\in{\bf B}_{\dR}^+$.

\medskip
One can then view $\eta_f^\alpha\otimes\omega_g\otimes\omega_h\otimes t_{r+2}$ as an element of $\Fil^0{\bf D}_{\dR}(V^*_N(f,g,h))$.

\begin{theorem}[\cite{BSV2020a}, Theorem A]
\label{BSVMuThmA}
In the above setting (in particular assuming that $f$ is $p$-ordinary), it holds that
\[
\log_{\BK}(\kappa_N(f,g,h))(\eta_f^\alpha\otimes\omega_g\otimes\omega_h\otimes t_{r+2})=_{E^\times}
\frac{\langle \Xi^{\ord}(g,h),f_\alpha^w\rangle_{Pet}}{\langle f_\alpha^w,f_\alpha^w\rangle_{Pet}},
\]
where:
\begin{enumerate}[(i)]
    \item $=_{E^\times}$ means equality up to an explicit scalar in $E^\times$
    \item $\Xi^{\ord}(g,h)$ is the classical modular form obtained applying Hida's ordinary projector to the $p$-adic modular form $d^{(k-l-m)/2} g^{[p]}\times h$ and  $d=q\tfrac{d}{dq}$ denotes Serre's derivative operator
    \item $f_\alpha^w$ denotes the ordinary $p$-stabilization of $f^w$
\end{enumerate}
\end{theorem}

\begin{rmk}
If instead of $\omega_g$ (resp. $\omega_h$), one considers $\omega_{\breve{g}}$ (resp. $\omega_{\breve{h}}$) for any $\breve{g}\in S_l(\Gamma_1(N),E)[g]$ (resp. $\breve{h}\in S_m(\Gamma_1(N),E)[h]$), then the same formula of Theorem \ref{BSVMuThmA} holds, i.e.,
\[
\log_{\BK}(\kappa_N(f,g,h))(\eta_f^\alpha\otimes\omega_{\breve{g}}\otimes\omega_{\breve{h}}\otimes t_{r+2})=_{E^\times}
\frac{\langle \Xi^{\ord}(\breve{g},\breve{h}),f_\alpha^w\rangle_{Pet}}{\langle f_\alpha^w,f_\alpha^w\rangle_{Pet}}.
\]
\end{rmk}

\subsection{Detecting non-vanishing of diagonal classes}
Now we specialize again to the setting of Section \ref{subsec:tameeulersystem}. Recall that this means:
\begin{enumerate}[(i)]
    \item $f\in S_k(\Gamma_0(N_f))$ is a $p$-ordinary newform (with $k\geq 2$ even);
    \item $g\in S_l(\Gamma_1(M),\chi_g,E)$ (resp. $h\in S_m(\Gamma_1(M),\chi_h,E)$) is the theta series attached to a Hecke character $\psi_1$ (resp. $\psi_2$) of a quadratic imaginary field $K$ of $\infty$-type $(1-l,0)$ (resp. $(1-m,0)$) and of conductor coprime to $N_f p$, with the requirement $\chi_g\chi_h=1$;
    \item $p$ is always inert in $K$;
    \item we keep assuming $l\geq m$ to fix ideas;
    \item we let $N=N_f\cdot M$.
\end{enumerate}

Moreover, we will work under Assumptions \ref{CRass} and \ref{CMass}.

\medskip
Let $\breve{g}=g(q^{N_f})$ and $\breve{h}=h(q^{N_f})$. We can consider a family of theta series $\Hg$ (resp. $\Hh$) as in Section \ref{subsection: theta families} having $g^{[p]}$ (resp. $h^{[p]}$) as specialization in weight $l$ (resp. $m$). We also let $\Hf$ denote the cuspidal Hida family which passes through the ordinary $p$-stabilization of $f$.

Then, following the construction of the $p$-adic $L$-function $\mathcal{L}_p^f(\Hf,\Hg,\Hh)$ carried out in \cite[Section 5.3]{Mar2024a}, it follows that specializing this $p$-adic $L$-function at weights $(k,l,m)$ one has
\[
\mathcal{L}_p^f(\Hf,\Hg,\Hh)(k,l,m)=_{E^\times}\frac{\langle \Xi^{\ord}(\breve{g},\breve{h}),f_\alpha^w\rangle_{Pet}}{\langle f_\alpha^w,f_\alpha^w\rangle_{Pet}}.
\]
We refer to \cite[Proposition 3.6 and Remark 3.7]{Mar2024a} for the description of the values of $\mathcal{L}_p^f(\Hf,\Hg,\Hh)$ in terms of ratios of Petersson products and we observe that the appearance of the \emph{test vectors} $\breve{g}$ and $\breve{h}$ is simply a consequence of the adjustments needed to obtain an optimal interpolation property for $\mathcal{L}_p^f(\Hf,\Hg,\Hh)$. Note that the triple $(k,l,m)$ lies \emph{outside} the interpolation region for $\mathcal{L}_p^f(\Hf,\Hg,\Hh)$.

\begin{defi}
Set $T(f,g,h):=T_{N_f}(f)\otimes_{\cl{O}}T_M(g)\otimes_{\cl{O}}T_M(h)(-1-r)$ and $V(f,g,h):=T(f,g,h)\otimes_{\cl{O}}E$. We define the class $\kappa(f,g,h)\in H^1(\Q,V(f,g,h))$ to be the cohomology class obtained from $\kappa_N(f,g,h)$ applying the projection from levels $(\Gamma_1(N),\Gamma_1(N),\Gamma_1(N))$ to levels $(\Gamma_1(N_f),\Gamma_1(M),\Gamma_1(M))$ that corresponds to the choice of test vectors $(f,\breve{g},\breve{h})$ as above.
\end{defi}

We also set dually $T^*(f,g,h):=T^*_{N_f}(f)\otimes_{\cl{O}}T^*_M(g)\otimes_{\cl{O}}T^*_M(h)(2+r)$ and $V^*(f,g,h)=T^*(f,g,h)\otimes_{\cl{O}}E$.
\begin{corollary}
In the above setting, it follows that $\mathcal{L}_p^f(\Hf,\Hg,\Hh)(k,l,m)\neq 0$ implies that $\kappa_N(f,g,h)\neq 0$ and that $\kappa(f,g,h)\neq 0$.
\begin{proof}
The fact that $\mathcal{L}_p^f(\Hf,\Hg,\Hh)(k,l,m)\neq 0$ implies $\kappa_N(f,g,h)\neq 0$ follows immediately from the previous discussion. For what concerns $\kappa(f,g,h)$, the adjointness between pullback and pushforward along degeneracy maps between modular curves (cf. \cite[eq. (188)]{BSVast1} for an instance) implies that 
\[
\log_{\BK}(\kappa_N(f,g,h))(\eta_f^\alpha\otimes\omega_{\breve{g}}\otimes\omega_{\breve{h}}\otimes t_{r+2})=\log_{\BK}(\kappa(f,g,h))(\eta_f^\alpha\otimes\omega_g\otimes\omega_h\otimes t_{r+2})
\]
where in the RHS we view $\eta_f^\alpha\otimes\omega_g\otimes\omega_h\otimes t_{r+2}\in\Fil^0{\bf D}_{\dR}(V^*(f,g,h))$. The claim concerning $\kappa(f,g,h)$ follows immediately.
\end{proof}
\end{corollary}

\subsection{Detecting non-vanishing of anticyclotomic Euler systems}
We keep working in the setting of the previous section. As already observed in Remark \ref{rmk: decompGalois}, we have a decomposition (with the same notation appearing there)
\[
T(f,g,h)\cong \big(T(f)(1-k/2)\otimes_{\cl{O}}\Ind_{K}^{\Q}\cl{O}(\hat{\eta_1}^{-1})\big)\oplus \big(T(f)(1-k/2)\otimes_{\cl{O}}\Ind_{K}^{\Q}\cl{O}(\hat{\eta_2}^{-1})\big)
\]
and similarly for the dual version
\[
T^*(f,g,h)\cong \big(T^*(f)(k/2)\otimes_{\cl{O}}\Ind_{K}^{\Q}\cl{O}(\hat{\eta_1})\big)\oplus \big(T^*(f)(k/2)\otimes_{\cl{O}}\Ind_{K}^{\Q}\cl{O}(\hat{\eta_2})\big)
\]
We denote $p^{?}_i$ for $?\in\{*,\emptyset\}$ and $i\in\{1,2\}$ the projections onto the $i$-th summand in the two cases. We denote in the same way the maps induced by such projections at the level of Galois cohomology and Dieudonn\'{e} modules.
\begin{lemma}
We have $p_i(\kappa(f,g,h))=\kappa_{f,\eta_i}$ (cf. Notation \ref{nota: bottom}).
\begin{proof}
This follows directly from the construction.
\end{proof}
\end{lemma}
\begin{lemma}
We have $p^*_2(\eta_f^\alpha\otimes\omega_g\otimes\omega_h\otimes t_{r+2})=0$.
\begin{proof}
We claim that
\begin{equation}
\label{claimV2*}
\Fil^0{\bf D}_{\dR}(V^*_2)= \Fil^1{\bf{D}}_{\dR}(V^*(f))\otimes_E {\bf D}_{\dR}((\Ind_K^{\Q}E(\hat{\eta_2}))(k/2)),   
\end{equation}
where $V^*_2:=V^*(f)(k/2)\otimes_E\Ind_{K}^{\Q}E(\hat{\eta_2})$.

Note that this allows us to conclude, since then $p^*_2(\eta_f^\alpha\otimes\omega_g\otimes\omega_h\otimes t_{r+2})$ would lie in
\[
\big(\Fil^0{\bf D}_{\dR}(V^*_2)\big)\cap \big({\bf D}_{\dR}(V^*(f))^{\varphi=\alpha_f}\otimes_E {\bf D}_{\dR}((\Ind_K^{\Q}E(\hat{\eta_2}))(k/2))\big)=\{0\}.
\]
To prove the claim \eqref{claimV2*}, it suffices to have a look at the Hodge-Tate weights involved. Recalling that $\eta_2$ has $\infty$-type $(-(l-m)/2,(l-m)/2)$, we see that (cf. the analogous discussion in the proof of Lemma \ref{lemma: descriptionSelmer}) the Hodge-Tate weights of $V_2^*$ are $1-k/2\pm (l-m)/2\leq 0$ and $k/2\pm (l-m)/2>0$. Since $\Fil^0{\bf D}_{\dR}(-)$ captures non-positive Hodge-Tate weights, we immediately deduce that $\Fil^0{\bf D}_{\dR}(V^*_2)$ must be the subspace of ${\bf D}_{\dR}(V^*_2)$ corresponding to the Hodge-Tate weights $1-k/2\pm (l-m)/2$, which is precisely 
\[
\Fil^1{\bf{D}}_{\dR}(V^*(f))\otimes_E {\bf D}_{\dR}((\Ind_K^{\Q}E(\hat{\eta_2}))(k/2)).
\]
\end{proof}
\end{lemma}

\begin{corollary}
\label{cor: restrtoeta1}
In the above setting, we have that
\[
\log_{\BK}(\kappa(f,g,h))(\eta_f^\alpha\otimes\omega_g\otimes\omega_h\otimes t_{r+2})=\log_{\BK}(\kappa_{f,\eta_1})(p^*_1(\eta_f^\alpha\otimes\omega_g\otimes\omega_h\otimes t_{r+2})).
\]
In particular $\mathcal{L}_p^f(\Hf,\Hg,\Hh)(k,l,m)\neq 0$ implies that $\kappa_{f,\eta_1}\neq 0$.
\end{corollary}

\section{The proof of Theorem \ref{mainThmIntro}}
\label{sec: final proof}
In this section we conclude the proof of Theorem \ref{mainThmIntro}, which is the main arithmetic application of this note.

\medskip
We keep the notation as in the previous section and we let $\chi$ be an anticyclotomic Hecke character of $K$ of conductor $c\cdot\cl{O}_K$ with $c\in\Z_{\geq 1}$ such that $(c, p\cdot d_K\cdot N_f)=1$ and $c$ is not divisible by primes inert in $K$ (cf. Assumption \ref{CMass} (iii)). We also assume that $\chi$ has $\infty$-type $(-j,j)$ with $j\geq k/2$.

\medskip
With the notation of Section \ref{factorizationsectionintro}, the Galois character associated with the $p$-adic avatar of $j$ can thus be viewed as a character $\hat{\chi}:\mathscr{G}_\infty\to E^\times$ (for a suitable choice of $E$ large enough finite extension of $\Q_p$). Fixing again a splitting $\mathscr{G}_\infty\cong\Delta_c\times\Gamma^-$ (where recall that $\Delta_c$ is a finite group and $\Gamma^{-}\cong\Z_p$ is identified with the Galois group of the anticyclotomic $\Z_p$-extension of $K$) and write $\hat{\chi}=(\chi_t,\chi^-)$ according to this decomposition.

\begin{ass}
\label{asschit}
We assume that $\chi_t$ can be written as $\chi_t=\gamma/\gamma^\mathrm{c}$ with $\gamma$ a ray class character of $K$ of conductor dividing $c\cdot\cl{O}_K$.
\end{ass}

\begin{rmk}
According to \cite[Lemma 6.9]{DR2017}, $\chi_t$ can always be written as $\chi_t=\gamma/\gamma^\mathbf{c}$ with $\gamma$ a ray class character of $K$ of conductor prime to $p\cdot N_f$. Hence the above assumption is simply imposing an extra condition on the conductor of such $\gamma$, which is a consequence of Assumption \ref{CMass} (iii). It is likely that with some extra work one could relax this condition.
\end{rmk}

According to \cite[Theorem D]{CH2018}, in our setting (cf. Assumption \ref{CMass} (ii)) and assuming $p>k-2$ (which we will do anyways for the applications) one can find an auxiliary ring class character of conductor a power of a prime $\ell\neq p$ (and such that $(\ell,N_f\cdot D_K\cdot c)=1$) split in $K$ such that 
\begin{equation}
\label{nonvanishLbeta}
L(f/K,\delta^2,k/2)\neq 0.    
\end{equation}

Consider now the ray class characters $\xi_1=\gamma\delta$ and $\xi_2=\gamma^{-\mathbf{c}}\delta^{-1}$ (so that $\xi_1\xi_2=\chi_t$ and $\xi_1\xi_2^\mathbf{c}=\delta^2$) and the generalized families of theta series  $\Hg=\boldsymbol{\theta}_{\xi_1}$ and $\Hh=\boldsymbol{\theta}_{\xi_2}$.

\begin{lemma}
\label{lemma: nonvanishiff}
In the above setting, we have that $\mathcal{L}_p^f(\Hf,\Hg,\Hh)(k,j+1,j+1)\neq 0$ if and only if $\Theta_\infty^{\mathrm{Heeg}}(\Hf,\chi_t)(k,\chi^-)\neq 0$.
\begin{proof}
This is a direct consequence of equation \eqref{nonvanishLbeta} and the factorization of Theorem \ref{factorizationthmintro}. Specializing the triple product $p$-adic $L$-function $\mathcal{L}_p^f(\Hf,\Hg,\Hh)$ at $(k,j+1,j+1)$ corresponds to specializing the RHS of the factorization of Theorem \ref{factorizationthmintro} at $(k,\chi^-,1)$ (cf. \cite[Lemma 5.21]{Mar2024a}). More precisely, this means that:
\begin{enumerate}[(i)]
    \item we are specializing $\Theta_\infty^{\mathrm{Heeg}}(\Hf,\chi_t)$ at $(k,\chi^{-})$, hence outside the range of interpolation as we chose $j\geq k/2$;
    \item we are specializing $\Theta_\infty^{\mathrm{Heeg}}(\Hf,1)\big)$ at $(k,\delta^2)$. 
\end{enumerate} 
Note that for this last factor factor we are in the range of $p$-adic interpolation and that the interpolation formula (as recalled, for instance, in \cite[Proposition 5.18]{Mar2024a}) makes sure that
\[
\Theta_\infty^{\mathrm{Heeg}}(\Hf,1)(k,\delta^2)=0\quad\Leftrightarrow\quad  L(f/K,\delta^2,k/2)= 0
\]
Finally, for the term $\mathscr{A}_{\Hf\Hg\Hh}$, recall that its non-vanishing properties are described in Remark \ref{factorizationrmk} (iii). The lemma follows combining these observations.  
\end{proof}
\end{lemma}

The proof of Theorem \ref{mainThmIntro} can be finally obtained combining Lemma \ref{lemma: nonvanishiff}, Corollary \ref{cor: restrtoeta1} and Corollary \ref{cor: rank1}.

\printbibliography

\end{document}